\documentclass[letterpaper,11pt]{amsart}
\usepackage[utf8]{inputenc}
\usepackage[english]{babel}

\usepackage{mathtools,amsthm,amssymb,bbm,enumitem,hyperref}
\usepackage{tikz-cd}
\usepackage{comment}

\usepackage[dvipsnames]{xcolor}

\numberwithin{equation}{section}

\newcommand{\set}[1]{\left\{#1\right\}}
\newcommand{\supp}{\operatorname{supp}}
\newcommand{\ran}{\operatorname{ran}}
\newcommand{\abs}[1]{\left|#1\right|}
\newcommand{\norm}[1]{\left\lVert#1\right\rVert}
\newcommand{\Aut}{\operatorname{Aut}}
\newcommand{\Ad}{\operatorname{Ad}}
\newcommand{\Span}{\operatorname{Span}}

\newtheorem{thm}{Theorem}[section]
\newtheorem{lemma}[thm]{Lemma}
\newtheorem{prop}[thm]{Proposition}
\newtheorem{cor}[thm]{Corollary}

\newtheorem{claim}[thm]{Claim}

\newtheorem{introthm}{Theorem}
\newtheorem{introqst}[introthm]{Question}

\theoremstyle{definition}
\newtheorem{rmk}[thm]{Remark}
\newtheorem{defn}[thm]{Definition}

\newtheorem{notation}[thm]{Notation}

\title{The generator problem for $C^*$-algebras with comparison}
\author{Ben Bouwen}
\address{\hskip-\parindent Ben Bouwen, Institute of Mathematics, Czech Academy of Sciences, Prague, Czech Republic}
        \email{bouwen@math.cas.cz}
\author{Jorge Castillejos}
\address{\hskip-\parindent Jorge Castillejos, Instituto de Matem\'aticas, Unidad Cuernavaca, Universidad Nacional Aut\'onoma de M\'exico, Av.\ Universidad s/n, 62210 Cuernavaca, Morelos, México.}
\email{jorge.castillejos@im.unam.mx}
\author{Jakub Curda}
\address{\hskip-\parindent Jakub Curda, Mathematical Institute, Andrew Wiles building, Woodstock Rd, OX2 6GG, Oxford, United Kingdom} 
\email{jakub.curda@maths.ox.ac.uk} 

\thanks{The author B. Bouwen was supported by DFF grant 1054-00094B, the Czech Academy of Sciences CAS (RVO 67985840) and the Czech Science Foundation grant GF25-15403K. The author J. Castillejos was supported by UNAM-PAPIIT IN108126. The author J. Curda was supported by EPSRC grant EP/X026647/1.\\ For the purpose of Open Access, the authors have applied a CC BY public copyright license to any Author Accepted Manuscript (AAM) version arising from this submission.}

\begin{document}

\maketitle

\begin{abstract}
    We show that certain classes of unital, simple, separable $C^*$-algebras with comparison are singly generated, including monotracial crossed products of actions on the Cantor space and tracially AF algebras. In particular, we show that simple, separable, tracially AF algebras have generator rank one.
\end{abstract}

\section{Introduction}

In 1967, Kadison presented his list of problems on von Neumann algebras (see \cite{Ge03}), which had a monumental impact on the field's research directions for decades to come.
One of these problems, the so-called generator problem, asks whether every separably acting von Neumann algebra is generated by a single element -- a question originally posed by Pearcy in \cite[Remarks.(2)]{Pearcy62}.
At present, there is a large body of work devoted to this problem for broad classes of von Neumann algebras. 
We will give a more detailed account of the progress further below, but let us note that the problem has been solved in the positive in its entirety after the very recent announcement of a complete solution for the $\mathrm{II}_{1}$ factor case by the use of LLMs. All other cases had already been solved in the past years and decades, and even in the $\mathrm{II}_1$ factor setting, many classes of von Neumann algebras were already known to be singly generated.

In this paper, we study the analogue of this problem for $C^*$-algebras, which has garnered significant attention in recent months in particular, as will be discussed later. 
Many different formulations exist, but we will be mainly concerned with the following:
\begin{introqst}[{The generator problem for $C^*$-algebras, \cite[Question~1.2]{thielwinter} and \cite[Problem~XCVI]{questions}}] \label{introqst:generatorproblem}
    Is every unital, simple, separable $C^*$-algebra singly generated?
\end{introqst}
Note that separability is a necessary condition: any $C^{*}$-algebra generated by a single element is separable. The simplicity assumption was added to exclude elementary counterexamples in the commutative setting, as the minimum number of generators for such a $C^*$-algebra is equal to the smallest $n$ for which its spectrum embeds into $\mathbb{C}^n$.
The unitality requirement is purely for convenience -- a $C^*$-algebra is singly generated if and only if its unitization is.
Just like in the von Neumann algebraic setting, there are many results showing that various classes of $C^*$-algebras are singly generated. However, the question was very recently (\cite{elliottlinniu}) shown to have a negative solution in general. As the counterexample constructed in \cite{elliottlinniu} fails certain regularity properties the analogues of which are automatic in the von Neumann algebraic setting, it would be of interest to investigate the generator problem further under one or more of these regularity assumptions.

In this paper, we study the generator problem for $C^*$-algebras with strict or dynamical comparison (see Definitions \ref{def:strictcomp} \& \ref{def:dynamicalcomp}), solving it in the positive for several classes of $C^*$-algebras, including: 
\begin{enumerate}
    
    \item 
    simple, monotracial $C^*$-algebras with strict or dynamical comparison containing a regular abelian subalgebra with spectrum homeomorphic to the Cantor space (Theorem~\ref{thm:singlegeneration});
    \item 
    certain free products of $C^*$-algebras amalgamated over a regular copy of the continuous functions on the Cantor space (Corollary~\ref{cor:amalgams});
    \item 
    simple tracially AF algebras (Theorem~\ref{main.thm}).
    
\end{enumerate}

New examples of singly generated $C^*$-algebras coming from the results in this paper are the non-$\mathcal {Z}$-stable and non-nuclear tracially AF algebras constructed in \cite{NW21}, certain crossed products of profinite actions (see (\ref{sg_profinite}) below Corollary~\ref{cor:sg_crossedprod}), and in particular the non-nuclear Bunce--Deddens algebras recently introduced by Bell in \cite{bell} (see (\ref{sg_buncededdens}) below Corollary \ref{cor:sg_crossedprod}). 

We will discuss our results in more detail later. First, we give an overview of the current state of the art, both in von Neumann and $C^*$-algebraic settings.

\subsection{Singly generated operator algebras}
On the von Neumann algebraic side, separably acting type $\mathrm{I}$ and separably acting properly infinite von Neumann algebras were shown to be singly generated in \cite{Pearcy62} and \cite{Wogen69}, respectively.
In \cite{Willig74}, Willig reduced the problem to factors, thus leaving only the $\mathrm{II}_{1}$ factor setting open. Very recently, a complete solution for this case by use of LLMs was announced and made public in \cite{openai}. As such, the generator problem for von Neumann algebras is solved in its entirety in the positive. Even before that point, many classes of separably acting $\mathrm{II}_{1}$ factors had already been shown to be singly generated, including those containing a Cartan subalgebra (\cite{Popa85}), those with property $\Gamma$, and those with non-trivial tensor product decomposition (\cite{GePopa98}). In \cite{Shen}, Shen presented a new approach inspired by Voiculescu's free entropy dimension, covering the above classes as well as providing new examples. We refer the reader to \cite{Sri} for an overview of the results achieved before the announcement \cite{openai}.

Let us now turn to the $C^*$-algebraic setting.
Generalizing early results (\cite{Topping68,OlsenZame76,LiShen08}), Thiel and Winter proved in \cite[Theorem~3.8]{thielwinter} that every unital, separable, $\mathcal{Z}$-stable $C^*$-algebra is singly generated.
Moreover, they show that unital, separable $C^*$-algebras with sufficiently non-trivial tensor product decomposition are singly generated (\cite[Theorem~3.5]{thielwinter}).
An unpublished result of Kirchberg shows that any unital, separable, properly infinite $C^*$-algebra is singly generated (see \cite[Theorem~2.4]{thielwinter}), \cite{Thi21} shows that separable AF algebras are singly generated, and more recently, \cite{graphs} shows that separable graph $C^*$-algebras are singly generated.
Following Thiel and Winter's work, the focus has shifted towards $C^*$-algebras which are not $\mathcal{Z}$-stable. As a particular example, the reduced group $C^*$-algebra of the free group with two generators was recently shown to be singly generated in \cite{freegroup}.

Recently, the class of Villadsen algebras has garnered particular attention. These non-$\mathcal{Z}$-stable AH algebras exhibit exotic topological behavior used to provide pivotal counterexamples in $C^*$-theory, notably in the classification program for $C^*$-algebras (see \cite{Toms1} and \cite{Toms2}). When it comes to the generator problem, the result \cite{villadsen} shows that Villadsen algebras of the first type are singly generated. More recently, Question \ref{introqst:generatorproblem} was shown to have a negative answer in general by \cite{elliottlinniu}; the counterexample they construct, lies within the class of Villadsen algebras of the second type.

Shortly after, \cite{hayesWbundle} gave a negative solution for a version of the generator problem for $W^*$-bundles, a class of objects between tracial $C^*$-algebras and tracial von Neumann algebras.

Our results, featuring within the generator problem for $C^*$-algebras, reach beyond the $\mathcal{Z}$-stable setting, while maintaining a regularity assumption not satisfied by Villadsen algebras, namely the notion of comparison.

\subsection{The generator problem with a Cantor spectrum subalgebra}

Our first result was inspired by Shen's approach (\cite{Shen}) mentioned previously. Given a diffuse, finite von Neumann algebra $\mathcal{M}$, he defines a quantity which can be thought of as the area of a given element of $\mathcal{M}$. Based on this, he introduces an invariant $\mathcal{G}(\mathcal{M})$, which is the smallest possible area that any set of elements generating $\mathcal{M}$ can achieve. It is then shown that any type $\mathrm{II}_1$ factor $\mathcal{M}$ is singly generated if $\mathcal{G}(\mathcal{M})$ is small enough (\cite[Theorem 4.2]{Shen}). Subsequently, this criterion is checked for several classes of type $\mathrm{II}_1$ factors to obtain single generation (\cite[Corollary 5.16]{Shen}). 

In order to apply the same philosophy to obtain analogous results in the $C^*$-setting, we first introduce the $C^*$-algebraic analogue (Definition \ref{def:area}) of the aforementioned notion of area. One of the crucial (albeit elementary) technical tools underlying Shen's approach is the ability to construct matrix algebras of arbitrary size in $\mathcal{M}$. While this holds automatically in his setting, we need some additional assumptions to obtain an analogue of this tool for $C^*$-algebras. The first is the presence of a unital, diffuse (with respect to a given trace; see Definition \ref{def:diffuse}) copy of the continuous functions the Cantor space; this guarantees the existence of arbitrarily small, mutually orthogonal projections. The second is a notion of comparison (see Definitions \ref{def:strictcomp} and \ref{def:dynamicalcomp}), which allows us to construct matrix algebras from these projections. The latter is the most fundamental one of the two.

Under these assumptions, we prove a general theorem giving a positive answer to the generator problem provided there exists a countable set of generators with small area (Theorem \ref{thm:singlegeneration}). To obtain concrete example classes using this theorem, we show that any $C^*$-algebra with a faithful trace containing $C(X)$ as a diffuse, regular subalgebra, is generated by elements with small area (Corollary \ref{cor:smallgenerators}), and that elements which commute with $C(X)$ have small area (Corollary~\ref{cor:commuting_smallarea}). As such, we obtain single generation for the following classes of $C^*$-algebras.

\begin{introthm} \label{introthm:examples}
    Let $X$ be the Cantor space. The following classes of $C^*$-algebras are singly generated:
    \begin{enumerate}
        \item crossed products $C(X) \rtimes_r G$ with a faithful trace, with $G$ a countable discrete group containing a subgroup $F$ such that the restricted action induces a simple, monotracial crossed product $C(X) \rtimes_r F$ with strict or dynamical comparison (Corollary \ref{cor:sg_crossedprod});
        \item selfless crossed products of actions on the Cantor space, and in particular the non-nuclear Bunce--Deddens algebras introduced in \cite{bell} (Examples (\ref{sg_selfless}) \& (\ref{sg_buncededdens}) below Corollary \ref{cor:sg_crossedprod});
        \item amalgamated free products over a regular copy of $C(X)$ where one factor is simple and monotracial and has strict comparison (Corollary~\ref{cor:amalgams}), provided all factors admit faithful tracial states which agree on $C(X)$;
        \item any tensor product of one of the above with any separable, unital $C^*$-algebra (Corollary~\ref{cor:sg_tensorprod}).
    \end{enumerate}
\end{introthm}

We would like to highlight that these results extend beyond the $\mathcal{Z}$-stable setting. While the work of Matui--Sato \cite{matuisato} on the Toms--Winter conjecture implies that unital, separable, nuclear $C^*$-algebras with a unique trace and strict comparison are $\mathcal{Z}$-stable, this is false outside of the nuclear setting in general. Concrete counterexamples of this which fit within our framework, can be found in (2) and (3) of Theorem~\ref{introthm:examples}; see the discussion below Corollary~\ref{cor:sg_crossedprod}. 

\subsection{The generator problem for tracially AF algebras}

Lin's class of tracially approximately finite dimensional (tracially AF) $C^*$-algebras was inspired by Popa's approximation techniques of $C^*$-algebras via approximately central finite-dimensional corners. 
In contrast with separable AF algebras, which ask for finite-dimensional approximations in norm, Popa's approximation condition requires the existence of finite-dimensional subalgebras with unit approximately central for a finite set and with corresponding compressions close to these finite-dimensional subalgebras.
Lin's tracial version of this idea additionally requires the complementary corner to be small in a tracial sense. 
Although tracial approximation provides excellent control over a large finite-dimensional corner, it leaves an uncontrolled complementary corner which need not be small in norm. Since generation is a norm-theoretic property, it is not clear a priori how to recover global generation from such approximations.

Tracially AF algebras formed an important early class in the classification program, where in the nuclear setting they admit strong classification theorems and encompass natural examples arising from dynamics (\cite{LP10}). At the same time, tracial finite-dimensional approximation does not rely on nuclearity, and indeed, there are non-nuclear tracially AF algebras. Thus, the tracially AF condition isolates a finite-dimensional approximation phenomenon that is interesting independently of classification. 
Our theorem shows that this approximation property has consequences for the generator problem.
In fact, we can show a stronger form of generation: tracially AF algebras have \emph{generator rank} at most one, so that single generators are in fact dense.

\begin{introthm}
    The class of simple separable tracially AF algebras has generator rank at most one. In particular, every simple separable tracially AF algebra is singly generated. 
\end{introthm}

As shown in \cite{NW21}, there are tracially AF algebras that are both non-nuclear and non-$\mathcal{Z}$-stable. Thus, our theorem also yields new examples of singly generated algebras which are neither nuclear nor $\mathcal{Z}$-stable.

The proof of this theorem is heavily based on the additional matrix structure provided by a stronger form of the tracial approximations (see Proposition \ref{prop:esstracZstab}). 
Strict comparison is a key feature of tracially AF algebras, which allows us to produce partial isometries encoding the uncontrolled complementary corner. By carefully perturbing a given pair of self-adjoint elements, we can then recover both corners from the perturbed pair and hence obtain generators. In this way, the tracially small remainder, although potentially large in norm, can be encoded into arbitrarily small perturbations of the original pair.

\subsection*{New developments}
Since the release of the first version of this paper, there have been further developments related to the generator problem for $C^*$-algebras. 
More specifically, the works \cite{elliottliniu_crossed_product,elliottliniu_comparison,ma} also study the generator problem in the setting of comparison and/or for crossed product $C^*$-algebras.
For transparency, we describe the overlap with our work.
Our results pertaining to crossed products (Corollary~\ref{cor:sg_crossedprod} and onwards) are mainly focused on actions of non-amenable groups and therefore do not overlap with results of \cite{elliottliniu_crossed_product}.
The results in \cite{elliottliniu_comparison} do overlap with Theorem~\ref{thm:singlegeneration} of our paper in the setting of strict comparison. However, our theorem does not require strict comparison of the entire $C^*$-algebra, and is also applicable for dynamical comparison.
Finally, (unital) tracially AF algebras have strict comparison and CPoU\footnote{Indeed, unital tracially AF algebras are tracially $\mathcal Z$-stable by \cite[Proposition 5.3]{FL22} and \cite[Theorem 4.11]{FLL22}. Then, by \cite[Proposition 4.13]{FLL22}, this is equivalent to being uniformly McDuff (see also \cite[Proposition 4.6]{CLS}). This yields CPoU by \cite[Theorem 1.4]{traciallycomplete}.} and therefore fall within the scope of both \cite{ma} and \cite{elliottliniu_comparison}.

\subsection*{Historical note}
This project started in late 2024 when, in the last year of his PhD, Bouwen visited Oxford and started working with Curda, who was a second year PhD student at the time. 
Most of the contents of Section~\ref{sec:cantor} was established during this period.
At the same time, Castillejos was using similar techniques to obtain single generation of tracially AF $C^*$-algebras.
The two projects merged in the spring of 2025 and the results were finalized in the summer of 2026.

\subsection*{AI statement} AI was used for reference search only. All mathematical ideas and proofs are due to the human authors.

\subsection*{Acknowledgments} We thank Stuart White for bringing the generator problem to our attention and for valuable discussions. 
We would also like to thank Hannes Thiel for valuable discussions on the topic of single generation.
Part of this work was carried out during the NSF/CBMS Regional Conference in the Mathematical Sciences “Classifying amenable operator algebras,” held at TCU (2025). We are grateful to the organizers for their hospitality and for a stimulating meeting.

\section{Preliminaries}

\subsection{Normalizers}
In this section, we briefly discuss normalizers and some well-known results related to them. These will come into play in \S \ref{sec:cantor}.

\begin{defn}
    Let $D$ be a subalgebra of a $C^*$-algebra $A$. A \emph{normalizer} of $D$ in $A$ is an element $a\in A$ such that
    \begin{equation}
        aDa^*\cup a^*Da\subseteq D.
    \end{equation}
    Denote by $N_A(D)$ the set of normalizers of $D$ in $A$. Then $D$ is said to be \emph{regular} in $A$ if $C^*(N_A(D))=A$.
    
    If $D$ is a subalgebra of a von Neumann algebra $\mathcal{M}$, then it is said to be \emph{(von Neumann) regular} in $\mathcal{M}$ if the set $N_{\mathcal{M}}(D)\cap \mathcal{U}(\mathcal{M})$ of unitary normalizers generates $\mathcal{M}$ as a von Neumann algebra.
\end{defn}

The set $N_A(D)$ is closed under multiplication and adjoints, and contains $D$. Moreover, for any $a\in N_A(D)$, we have $a^*a\in D$. 

\begin{rmk} \label{rmk:CvsvNregular}
Note that von Neumann regularity is a stronger property compared to its $C^*$-algebraic counterpart in general.
However, if $\mathcal{M}$ admits a faithful trace and $D$ is an abelian subalgebra of $\mathcal{M}$, then it suffices for $N_{\mathcal{M}}(D)$ to generate $\mathcal{M}$ as a von Neumann algebra; this follows from \cite[Lemma~12.1.2]{Delaroche-Popa} and polar decomposition.
In particular, if $D$ is an abelian, ($C^*$)-regular $C^*$-subalgebra of a $C^*$-algebra $A$ and $A$ has a faithful trace $\tau$ with corresponding GNS representation $\pi_\tau$, then $\pi_\tau(D)''$ is von Neumann regular in $\pi_\tau(A)''$.
\end{rmk}

\begin{defn}
    Let $X$ be a compact Hausdorff space and let $A$ be a $C^*$-algebra containing $C(X)$ as a $C^*$-subalgebra. For a normalizer $a\in N_{A}(C(X))$, write
    \begin{equation}
    \begin{aligned}
        \supp(a)&:=\{x\in X\mid a^{*}a(x)>0\},\\
        \ran(a)&:=\supp(a^{*}).
    \end{aligned}
    \end{equation}
    Then $a$ is said to be \emph{clopen} if $\supp(a)$ and $\ran(a)$ are compact open subsets of $X$. We denote by $N_{A}^{c}(C(X))$ the set of clopen normalizers of $C(X)$ in $A$.
\end{defn}

In the above setting, any $a \in N_A(C(X))$ induces a partial homeomorphism $\alpha_{a}$ of $X$ from the support of $a$ to its range. 
By \cite[Proposition~4.7]{Renault} (see also \cite[Proposition~11.1.15]{simsszabowilliams}), $\alpha_{a}$ is characterized by the equation
\begin{equation}
    a^{*}fa(x)=f(\alpha_{a}(x))a^{*}a(x)
\end{equation}
for all $f\in C(X)$ and $x\in X$. Note also that $\alpha_{a^{*}}=\alpha_{a}^{-1}$. Since $\alpha_{a}$ is a partial homeomorphism, $a$ is clopen if $\supp(a)$ or $\ran(a)$ is clopen. Moreover, any partial isometry $v\in N_A(C(X))$ is automatically clopen.

The following lemma allows us to restrict our attention to normalizing partial isometries to generate $A$ when $X$ is zero-dimensional. The results are well-known, but we provide a short proof for the convenience of the reader.
We will use $\mathbbm 1_Y$ to denote the indicator function of a set $Y$.

\begin{lemma}
\label{lemma:clopen-normalizers}
     Let $X$ be a zero-dimensional, compact, Hausdorff space and let $A$ be a $C^*$-algebra containing $C(X)$ as a $C^*$-subalgebra. 
     The set $N_{A}^{c}(C(X))$ is dense in $N_{A}(C(X))$. Moreover, any $a\in N_A^c(C(X))$ has a polar decomposition $a=v\abs{a}$ in $A$, with $\abs{a}\in C(X)$ and $v\in N_A^c(C(X))$ a partial isometry for which $v^*v=\mathbbm{1}_{\supp(a)}$ and $vv^*=\mathbbm{1}_{\ran(a)}$.
\end{lemma}
\begin{proof}
    To prove the first part of the statement, let $a\in N_{A}(C(X))$ and $\varepsilon>0$. 
    Since $X$ is zero-dimensional, we can find a clopen $K\subseteq\supp(a)$ such that 
    \begin{equation}
        \sup_{x\in X\setminus K} \left| a^{*}a(x)\right| \leq\varepsilon .
    \end{equation}
    Let $p:=\mathbbm{1}_{K}$. Then $ap\in N_{A}^c(C(X))$ satisfies
    \begin{equation}
    \begin{aligned}
        \norm{ a-ap}&\leq\varepsilon\\
        \supp(ap)&=K\cap\supp(a)=K
    \end{aligned}
    \end{equation}
    since $\alpha_{ap}=\alpha_{a}|_{K}$.

    For the second part of the statement, note that since $\supp(a)$ is clopen, there exists $\varepsilon>0$ such that $a^*a(x)\geq \varepsilon$ for all $x\in \supp(a)$.
    Let $f(x):=\frac{1}{a^*a(x)}$ for $x\in \supp(a)$ and $f(x)=1$ otherwise. 
    Notice that $f$ is a continuous non-negative function on $X$.
    Let $v:=af^{1/2}\in A$.
    Then one computes that $v^*v=fa^*a=\mathbbm{1}_{\supp(a)}$ and so $v$ is a partial isometry.
    One easily checks that $v$ belongs to the normalizer, provides a polar decomposition of $a$, and the required equalities hold.
\end{proof}

We also record the following lemma. For a proof, see \cite[Proposition~1.6]{Kumjian} and \cite[Lemma~12.1.2]{Delaroche-Popa}.

\begin{lemma}
\label{lemma:npi_extends}
    Let $A$ be $C^*$-algebra with a faithful trace $\tau$ and containing $C(X)$, with $X$ a zero-dimensional, compact, Hausdorff space. Write $\mathcal{M} := \pi_\tau(A)''$ and consider $ L^{\infty}(X,\mu_\tau)=\pi_\tau(C(X))''\subseteq \mathcal{M} $ with $\pi_\tau$ the GNS representation associated to $\tau$ and $\mu_\tau$ the measure on $X$ induced by $\tau$. For any $v\in N_A(C(X))$, there exists a unitary $u\in N_{\mathcal{M}}(L^{\infty}(X,\mu_\tau))$ such that $vv^{*}uv^{*}v=v$.
\end{lemma}

\subsection{Traces and comparison}

By the word trace, we will always implicitly mean a tracial state unless specified otherwise, and by ``monotracial'' or ``with a unique trace'' we will always mean having a unique tracial state.

Recall that for two positive elements $a$ and $b$ of a $C^*$-algebra $B$, $a$ is said to be \emph{Cuntz subequivalent} to $b$ if for every $\varepsilon >0$, there exists $r \in B$ such that $\|r^*br - a\| <\varepsilon$; we denote this by $a \precsim b$. Two projections $p,q\in B$ are called \emph{Murray-von Neumann subequivalent} if there is an element $v\in B$ such that $v^*v=p$ and $vv^*\leq q$. We will also denote this by $p\precsim q$, as this is equivalent to Cuntz subequivalence for projections (see \cite[Lemma~2.8]{gardellaperera}).

\begin{defn} \label{def:strictcomp}
    Let $B$ be a $C^*$-algebra with non-empty tracial state space $T(B)$. For $b \in B$ and $\tau \in T(B)$, we write
    \begin{equation}
        d_\tau(b) := \lim_{k\to\infty} \tau(b^{1/k}),
    \end{equation}
    where by abuse of notation we also write $\tau$ for the non-normalized extension of $\tau$ to $B\otimes \mathcal{K}$. Then $B$ is said to have \emph{strict comparison (of positive elements by bounded traces)} if for arbitrary positive elements $a,b\in (B\otimes \mathcal{K})_+$ such that 
    \begin{equation}
        d_\tau(a) < d_\tau(b)
    \end{equation}
    for all $\tau \in T(B)$, we have $a \precsim b$.
\end{defn}

Note that for a projection $p \in B$, we have $d_\tau(p) = \tau(p)$, so strict comparison for projections says exactly that Murray-von Neumann subequivalence can be detected by strict inequality on all traces of the respective projections. In fact, the precise statement we will use for the entirety of this paper is the following: if for two projections $p, q \in B$, we have $\tau(p) < \tau(q)$ for all $\tau \in T(B)$, then $p\precsim q$.

The results in \S \ref{sec:cantor} also pertain to the following setting, a version of dynamical comparison for abelian subalgebras which resembles \cite[Definition~1.5]{kopsacheiliswinter}.
In our case, the abelian subalgebra is not assumed to be regular.

\begin{defn} \label{def:dynamicalcomp}
    Let $B$ be a $C^*$-algebra with non-empty trace space and $D$ an abelian subalgebra of $B$.
    We say that $D$ has \emph{dynamical comparison} relative to $B$ if for any two projections $p,q\in D$ such that $\tau(p)<\tau(q)$ for all $\tau\in T(B)$, we have $p\precsim q$ by a partial isometry $v\in N_B(D)$.
\end{defn}

For our purposes, the spectrum of the subalgebra will always be zero-dimensional. Note that the definition of dynamical comparison for zero-dimensional dynamical systems from \cite[Definition~3.1~and~Proposition~3.5]{kerr} implies ours.

The following notion and subsequent results will be used in \S \ref{sec:cantor}.

\begin{defn}[\cite{PedersenDiffuse, ThielDiffuse}] \label{def:diffuse}
    A trace $\tau$ on a $C^*$-algebra $A$ is said to be \emph{diffuse} if the associated GNS completion $\pi_\tau(A)''$ is a diffuse von Neumann algebra, i.e.\ has no minimal projections.

    If $D$ is a $C^*$-subalgebra of $A$, we say $\tau$ is \emph{diffuse on $D$} if $\tau|_D$ is a diffuse trace on $D$.
\end{defn}

In the particular case where $A$ is abelian, a trace $\tau$ on $A$ is diffuse if and only if the induced measure on the spectrum of $A$ is diffuse, i.e.\ has no atoms (see for example \cite[Example~3.3]{ThielDiffuse}).
The next lemma is likely well-known. 
We were not able to locate a precise reference so we provide a short proof.

\begin{lemma}
\label{lemma:no-atoms-in-vna}
    Let $(\mathcal{M},\tau) $ be a $\mathrm{II}_{1}$ factor and $D\subseteq \mathcal{M} $ a regular, abelian subalgebra.
    Then $D$ is either finite-dimensional or diffuse.
\end{lemma}

\begin{proof}
    We first show that $D$ is either atomic or diffuse. To this end, write $D=L^\infty(X)\oplus \ell^\infty(Y)$ with $X$ a (possibly empty) diffuse measure space and $Y$ a (possibly empty) countable set.
    Denote the unit of $L^\infty(X)$ by $p:=1_{L^\infty(X)}$.
    
    Pick an arbitrary (possibly zero) partial isometry $v\in N_\mathcal{M}(D) \cap p \mathcal{M}(1-p)$ and an arbitrary (possibly zero) minimal projection $q\leq v^*v \in \ell^\infty(Y)$ in $\ell^\infty(Y)$. 
    By \cite[p.~192]{Delaroche-Popa}, the corner $q\ell^\infty(Y)q\cong \mathbb{C} $ is isomorphic to the corner $vqv^*L^\infty(X)vqv^*$.
    Hence, $vqv^*$ is a minimal projection in the diffuse algebra $L^\infty(X)$, which is only possible if $vqv^*=0$. It follows that
    \begin{equation}
        0 = v^*(vqv^*)v = q,
    \end{equation}
    so $0$ is the only minimal subprojection of $v^*v$ in $\ell^\infty(Y)$. Therefore, $v$ must be zero.
    This implies that either $p=0$ or $p=1$, as otherwise $N_{\mathcal{M}}(D)$ would be contained in the proper (by factoriality of $\mathcal{M}$) von Neumann subalgebra $p\mathcal{M} p\oplus (1-p)\mathcal{M} (1-p)$ of $\mathcal{M}$, contradicting regularity of $D$.
    The subalgebra $D$ is therefore diffuse or atomic.

    Now suppose that $D$ is atomic, and let $P$ be the (countable) set of minimal projections in $D$.
    Consider the action $\alpha: U(N_{\mathcal{M}}(D))\curvearrowright P$ by unitary conjugation. As $P$ is countable, $\alpha$ is a permutation of $P$.
    Let $S\subseteq P$ be an arbitrary $\alpha$-invariant subset of $P$.
    Then for all $u\in U(N_{\mathcal{M}}(D))$, $\alpha(u)$ fixes $S$, so $u$ commutes with the indicator function $\mathbbm{1}_{S}$ of $S$.
    By regularity of $D$, it follows that $\mathbbm{1} _S$ is central in $\mathcal{M} $ and thus either 0 or 1.
    Consequently, the action $\alpha$ is transitive on $P$ and since it is furthermore trace-preserving, $\tau(p)=\tau(q)$ for all $p,q\in P$.
    Hence $P$ must be finite and so $D$ is finite-dimensional.
\end{proof}

\begin{cor}
\label{cor:diffuse}
    Let $B$ be a $C^*$-algebra containing $C(X)$ as a regular $C^*$-subalgebra, with $X$ an infinite, compact Hausdorff space. Any faithful, extremal trace $\tau$ on $B$ is diffuse on $C(X)$.
\end{cor}
\begin{proof}
Let $\pi_\tau$ be the GNS representation of $B$ associated to $\tau$ and $\mu_\tau$ the measure on $X$ induced by $\tau$.
Since $X$ is infinite and $\tau$ is a faithful, extremal trace, it follows that $\mathcal{M} := \pi_\tau(B)''$ is a $\mathrm{II}_1$ factor by \cite[Theorem 6.7.3]{Dixmier}. 
As $B$ contains $C(X)$ as a regular $C^*$-subalgebra, it follows from Remark~\ref{rmk:CvsvNregular} that $L^\infty(X,\mu_{\tau})$ is a regular von Neumann subalgebra of $\mathcal{M}$.
From the previous lemma, it then follows that $L^\infty(X,\mu_{\tau})$ is either finite-dimensional or diffuse. However, it cannot be finite-dimensional, as $\tau$ is faithful and $X$ is infinite. This concludes the proof.
\end{proof}

\subsection{Generator rank}

In \cite{Thi21, Thiel22, Thiel23}, Thiel introduced and developed the notion of generator rank of $C^*$-algebras. It provides a general framework for establishing single generation of $C^*$-algebras.
	
\begin{defn}[{\cite[Definitions 2.1, 3.1]{Thi21}}]
		Let $A$ be a $C^*$-algebra. Then $gr_0(A)$ is defined as the smallest integer $n \geq 0$ such that for every $a_0, \ldots, a_n \in A_{sa}, \varepsilon> 0$ and $c\in A$, there exist $b_0, \ldots, b_n \in A_{sa}$ such that 
		\begin{equation}
			\mathrm{dist}(c, \ C^*(b_0, \ldots, b_n))< \varepsilon, \qquad \| b_i - a_i \| < \varepsilon, \quad i=0, \ldots, n.
		\end{equation}
		If no such $n$ exists, then $\mathrm{gr}_0(A) = \infty$. 
		The \emph{generator rank} of $A$ is defined as 
		\begin{equation}
			\mathrm{gr}(A):= \mathrm{gr}_0 (A^\sim),
		\end{equation}
        where $A^\sim $ is the minimal unitization of $A$.
	\end{defn}

It follows from the definition that if $A$ is unital, then $\mathrm{gr}(A)$ and $\mathrm{gr}_0(A)$ agree. 
The following theorem gives a key characterization of the generator rank in terms of the density of generating tuples. 
Roughly speaking, a $C^*$-algebra $A$ has generator rank at most $n$ if the set of $n$-tuples of self-adjoint elements that generate $A$ is dense in $A_{sa}^{\oplus n}$.
    
\begin{thm}[{\cite[Theorem 3.4]{Thi21}}]
	Let $A$ be a separable $C^*$-algebra and $n \in \mathbb{N}$. Then $\mathrm{gr}(A) \leq n$ if and only if for every $a_0, \ldots, a_n \in A_{sa}$ and $\varepsilon>0$, there exist $b_0,\ldots, b_n \in A_{sa}$ such that $A = C^*(b_0, \ldots, b_n)$ and
	\begin{equation}
		\|b_i -a_i\| < \varepsilon, \qquad i=0, \ldots, n.
	\end{equation} 
\end{thm}

Since any element is the sum of two self-adjoint elements, it follows that generator rank at most $1$ implies single generation and that the generator is \emph{generic}; i.e. the set of generators is dense in the algebra.	
The generator rank is a dimension theory with useful permanence properties. We only record here the following proposition regarding the generator rank of inductive limits, which will be used in \S \ref{sec:TAF}.

\begin{prop}[{\cite[Theorem 6.3]{Thi21}}]\label{thm:gr.ind.limits}
	Let $A$ be an inductive limit of a sequence $(A_\lambda)$ of $C^*$-algebras. Then
	\begin{equation}
		\mathrm{gr}(A) \leq  \liminf_\lambda \mathrm{gr}(A_\lambda).
	\end{equation}
\end{prop}

\section{$C^*$-algebras with a Cantor spectrum subalgebra} \label{sec:cantor}

We now turn to the study of single generation of $C^*$-algebras with a regular, abelian subalgebra with spectrum homeomorphic to the Cantor space.
The overarching idea is the following: we use tracially large matrix subalgebras to carefully assemble a pair of generating self-adjoints. 
In particular, we identify suitable families of pairwise orthogonal, pairwise Murray-von Neumann equivalent projections, and thus copies of matrix algebras.

\subsection{The area and support of an element}
\begin{defn}
    Let $A$ be a unital $C^*$-algebra. A \emph{partition of unity} for $A$ is a finite set $\{p_i\}_{i=1}^n$ of non-zero, mutually orthogonal projections in $A$ which add up to $1_A$.
    
    A \emph{refinement} of $\{p_i\}_{i=1}^n$ is another partition of unity $\{p_{i,j}\mid 1\leq i\leq n, 1\leq j \leq n_i\}$ for $A$ such that $p_i = \sum_{j=1}^{n_i} p_{i,j}$ for all $i$.
\end{defn}

\begin{notation}
    Let $a_1, ..., a_n$ and $b_1, ..., b_m$ be positive elements in a $C^*$-algebra $A$, with $n\geq m$ and the $b_i$ mutually orthogonal. We say the tuple $(a_i)_i$ is \emph{subordinate} to $(b_j)_j$ if the assignment
    \begin{equation}
        \{1, ..., n\} \to \{1, ..., m\}: i \mapsto j\quad \text{ with } j \text{ such that } a_i \leq b_j
    \end{equation}
    is a well-defined, surjective, order-preserving map. In other words, each $a_i$ is under precisely one $b_j$, each $b_j$ has at least one $a_i$ under it, and the $a_i$ are grouped by the $b_j$ they are under, in the same order as the $b_j$.
\end{notation}

\begin{rmk}
    Any tuple $(q_k)_k$ of mutually orthogonal projections which is subordinate to a partition of unity $(p_i)_i$, induces a refinement of $\{p_i\}_i$ by adding to $\{q_k\}_k$ the non-zero defects (if there are any) of each $p_i$ by the sum of all $q_k \leq p_i$. We nonetheless introduce the above notion to make notation in some proofs less cumbersome.
\end{rmk}

Inspired by Shen \cite{Shen}, we define the area and support of an element $x$ in a $C^*$-algebra $D$ with respect to a partition of unity. Note that in the case when all the projections have the same trace, this recovers Shen's definition \cite[Definition~2.1]{Shen}.

\begin{defn}
\label{def:area}
    Let $D$ be a unital $C^*$-algebra with a faithful trace $\tau$ and let $p_{1},\dots,p_{n} \in D$ be a partition of unity for $D$.
    The \emph{area} of an element $x \in D$ with respect to $\{p_{i}\}_{i=1}^{n}$ and $\tau$ is
    \begin{equation}
        \mathcal{I}(x\mid \{p_{i}\}_{i=1}^{n}):=\sum_{i,j:p_{i}xp_{j}\neq0}\tau(p_{i})\tau(p_{j}).
    \end{equation}
    The \emph{support} of $x$ with respect to $\{p_{i}\}_{i=1}^{n}$ and $\tau$ is
    \begin{equation}
        \mathcal{S}(x\mid \set{p_i}_{i=1}^n):=\sum_{ p_ix\neq0 \text{ or } xp_i\neq0} p_i.
    \end{equation}
\end{defn}

\begin{notation}
    We will also use the above notation in the situation where $D$ is unitally contained in a unital $C^*$-algebra $A$ and $x$ is an element of $A$. We will also omit ``with respect to $\{p_{i}\}_{i=1}^{n}$ and $\tau$'' in writing when this data is clear from the context.
    Moreover, if the indexing sets are clear from context, we will write $\{p_i\}$ for $\{p_{i}\}_{i=1}^{n}$.
\end{notation}

\begin{defn}
    \label{def:smallarea}
    Let $A$ be a unital $C^*$-algebra, and $D$ a unital $C^*$-subalgebra of $A$ with faithful trace $\tau$. We say an element $x\in A$ has \emph{small area} with respect to $(D,\tau)$ if for every $\varepsilon>0$, there exists a partition of unity $\set{p_i}_{i=1}^n$ for $D$ (i.e.\ each $p_i$ is in $D$) such that
    \begin{equation}
        \mathcal{I}(x\mid \set{p_i}_{i=1}^n) < \varepsilon.
    \end{equation}
\end{defn}

For this notion to be sensible, there need to exist partitions of unity for $D$ of projections with arbitrarily small trace.  This will be case for our purposes, in which $D$ is $C(X)$ with $X$ the Cantor space and $\tau$ is diffuse, as a consequence of Lemma~\ref{lemma:equalizing-projections}. For now, we will make the following elementary observations related to the area and support. 

\begin{lemma} \label{lemma:areaprops}
    Let $A$ be a unital $C^*$-algebra and $D \subseteq A$ a unital $C^*$-subalgebra with a faithful trace $\tau$. Let $\{p_i\}_{i=1}^n$ be a set of pairwise orthogonal projections in $D$ summing up to $p$ and let $x \in A$.
    \begin{enumerate}
        \item \label{lemma:areaprops_commute} If $x$ commutes with $\{p_i\}_{i=1}^n$, then
        \begin{equation}
        \mathcal{I}(x\mid \{p_i\})\leq n \max_{1\leq i\leq n}\tau(p_i)^2.
        \end{equation}
        \item \label{lemma:areaprops_refine} For any refinement $\{p_{i,j}\}_{1\leq i\leq n, 1\leq j\leq n_{i}}$ of $\{p_i\}_{i=1}^n$, we have
        \begin{equation}
        \mathcal{I}(pxp\mid \{p_{i,j}\}_{1\leq i\leq n, 1\leq j\leq n_{i}})\leq \mathcal{I}(pxp\mid \{p_{i}\}_{i=1}^{n}).
        \end{equation}
        \item \label{lemma:areaprops_onesided} For any $y,z \in pA$, we have
        \begin{equation}
        \mathcal{I}(yz^*\mid \{p_i\})\leq \tau(\mathcal{S}(yy^*\mid \{p_i\})) \tau(\mathcal{S}(zz^*\mid \{p_i\})).
        \end{equation}
    \end{enumerate}
\end{lemma}
\begin{proof}
(\ref{lemma:areaprops_commute}):
    As
    \begin{equation}
        0 \neq p_i x p_j = x p_i p_j = \delta_{ij} x p_i
    \end{equation}
    implies $i=j$ for all $i,j$, it follows that
    \begin{equation}
         \mathcal{I}(x\mid \{p_i\})\leq \sum_{i=j} \tau(p_i)\tau(p_j) \leq n \max_{1\leq i\leq n}\tau(p_i)^2.
    \end{equation}   

(\ref{lemma:areaprops_refine}):
    This follows immediately from the fact that
    \begin{equation}
         \sum_{k,l: p_{i,k}xp_{j,l}\neq 0} \tau(p_{i,k})\tau(p_{j,l}) \leq \sum_{k,l} \tau(p_{i,k})\tau(p_{j,l}) = \tau(p_i)\tau(p_j)
    \end{equation}
    for all $i$ and $j$.

(\ref{lemma:areaprops_onesided}):
    For each $i \in\{1,...,n\}$, we claim that $xx^*p_i\neq 0$ if and only if $p_i xx^*\neq 0$ if and only if $p_ix\neq 0$.
    The first equivalence is immediate, and so is the forward direction of the second, so suppose that $p_ixx^*=0$ for some $i$.
    Then $(p_ix)(p_ix)^*=p_ixx^*p_i=0$ and so $p_i x=0$.
    The same of course holds for $y$ in place of $x$.
    Hence
    \begin{equation}
        \begin{aligned}
            \mathcal{I}(xy^*\mid \{p_i\})&=\sum_{i,j:p_ixy^*p_j\neq0}\tau(p_i)\tau(p_j)\\
            &\leq \sum_{\substack{i:p_ix\neq0\\j:y^*p_j\neq 0}}\tau(p_i)\tau(p_j)\\
            &=\sum_{i:p_ixx^*\neq0\text{ or }xx^*p_i\neq 0}\tau(p_i)\sum_{j:p_jyy^*\neq0\text{ or }yy^*p_j\neq 0}\tau(p_j)\\
            &=\tau(\mathcal{S}(xx^*\mid\{p_i\}))\tau(\mathcal{S}(yy^*\mid \{p_i\})),
        \end{aligned}
    \end{equation}
    finishing the proof.
\end{proof}

Our methods rely on \cite[Theorem~4.1]{Shen}, which we recall next.
The theorem is stated for von Neumann algebras in \cite{Shen}, but the proofs apply\footnote{Note that in \cite[Theorem~4.1]{Shen}, the trace is only applied to the matrix projections belonging to the diagonal.} in the $C^*$-algebraic setting as well.

\begin{thm}[{\cite[Theorem~4.1]{Shen}}]
\label{thm:shen}
    Let $A$ be a unital $C^*$-algebra and let $B\subseteq A$ be a unital $C^*$-subalgebra with a faithful trace $\tau$.
    Suppose that $\set{e_{ij}}_{1\leq i,j\leq n}$ are matrix units of a unital copy of $M_n(\mathbb C)$ inside $A$ with $n\geq 4$ such that $e_{ii}\in B$ for all $i$.
    If $x_1,\dots,x_k$ are elements of $A$ such that
    \begin{equation}
        \sum_{i=1}^k\mathcal{I}(x_i\mid \set{e_{jj}}_{j=1}^n)=c^2
    \end{equation}
    with $c\leq 1/4$, then there is a projection $q\in C^*(x_1,\dots,x_k,\set{e_{ij}})$ such that
    \begin{equation}
        \tau(\mathcal{S}(q\mid \set{e_{ii}})\leq 2c + 2/n
    \end{equation}
    and 
    \begin{equation}
        C^*(q,\set{e_{ij}})=C^*(x_1,\dots,x_k,\set{e_{ij}}).
    \end{equation}
\end{thm}

Also recall that any $C^*$-algebra $A$ generated by countably many commuting projections has zero-dimensional spectrum, which then embeds into $[0,1]$, and so $A$ is singly generated by \cite[Proposition~2]{Nagisa04}:

\begin{lemma}
    \label{lemma:single-generation-with-projections}
    Let $A$ be a $C^*$-algebra generated by countably many commuting projections.
    Then $A$ is generated by a single self-adjoint element.
\end{lemma}

\subsection{Approximations}
In this section, we cover a couple of results about the $C^*$-algebra of continuous functions on the Cantor space based on approximations from its von Neumann completions. Most of these facts are folklore; we include their proofs for completeness.

Let $X$ be the Cantor space, and $\tau$ a faithful trace on $C(X)$, with corresponding measure $\mu_\tau$ on $X$ and von Neumann completion $L^\infty(X,\mu_\tau)$. For $x\in L^\infty(X,\mu_\tau)$, we denote its $2$-norm by
\begin{equation}
    \norm{x}_{2,\tau}:=\tau(x^*x)^{1/2}.
\end{equation}
Recall that for two projections $p,q\in L^\infty(X,\mu_\tau)$, Kadison's inequality gives 
\begin{equation}
    \abs{\tau(p)-\tau(q)}\leq \norm{p-q}_{2,\tau}.
\end{equation}

\begin{lemma}
    \label{lemma:approximating-projections}
    Let $X$ be the Cantor space and $\tau$ a faithful trace on $C(X)$. For any projection $p\in L^{\infty}(X,\mu_\tau)$ and $\varepsilon>0$, there exists a projection $\tilde{p}$ in $C(X)$ such that $\norm{ p-\tilde{p}}_{2,\tau}\leq\varepsilon$.
\end{lemma}

\begin{proof}
    Since projections in $C(X)$ correspond bijectively to clopen subsets and projections in $L^{\infty}(X)$ correspond bijectively to measurable sets modulo null sets, the statement is equivalent to the following: For every $Y\subseteq X$ measurable there is $\tilde{Y}\subseteq X$ clopen such that 
    \begin{equation}
         \mu_{\tau}(Y\Delta\tilde{Y})\leq\varepsilon.
    \end{equation}

    Here $\Delta$ denotes the symmetric difference of two sets.
    As $\mu_{\tau}$ is regular, $Y$ is measurable and $X$ is compact, there is compact $K\subseteq Y$ and an open set $U$ such that $Y \subseteq U$ and $\mu_{\tau}(U\setminus K)\leq \varepsilon$. Since $U$ is open and $X$ is zero-dimensional, $U$ is a (possibly infinite) union of clopens so by compactness of $K$, there are finitely many clopens $\left\{U_{i}\right\}_{i=1}^{n}$ such that 
    \begin{equation}
       K\subseteq\bigcup_{i=1}^{n}U_{i}\subseteq U. 
    \end{equation}
    Let $\tilde{Y}:=\bigcup_{i=1}^{n}U_{i}$ and observe that $Y\Delta\tilde{Y}\subseteq U\setminus K$ so $\mu_{\tau}(Y\Delta\tilde{Y})\leq\varepsilon$.
\end{proof}

We now show that the above can also be done for multiple projections while preserving orthogonality.

\begin{lemma}
\label{lemma:approximating-multiple-projections}
    Let $X$ be the Cantor space and $\tau$ a faithful trace on $C(X)$. Let $p_{1},\dots,p_{n}\in L^{\infty}(X,\mu_\tau)$ be pairwise orthogonal projections and let $\varepsilon>0$. Then there are pairwise orthogonal projections $\tilde{p}_{1},\dots,\tilde{p}_{n}$ in $C(X)$ such that $\norm{ p_{i}-\tilde{p}_{i}}_{2,\tau}\leq\varepsilon$ for each $i$.
\end{lemma}
\begin{proof}
    Set $\delta:=\frac{\varepsilon}{2n+1}$.
    Lemma~\ref{lemma:approximating-projections} provides projections $q_{1},\dots,q_{n}\in C(X)$ such that $\norm{q_{i}-p_{i}}_{2,\tau}\leq \delta$. 
    Now let 
    \begin{align}
        \tilde{p}_{1}&:=q_{1},\text{ and}\\
        \tilde{p}_{i}&:=q_{i}(1-\sum_{k=1}^{i-1}\tilde{p}_{k}).
    \end{align}
    Note that these are pairwise orthogonal by construction.
    Since \begin{align}\norm{q_iq_k}_{2,\tau}\leq 2\delta+\norm{p_ip_k}_{2,\tau}=2\delta,
    \end{align} 
    we estimate
    \begin{equation}
    \begin{aligned}
        \norm{ p_{i}-\tilde{p}_{i}}_{2,\tau}&\leq\norm{ p_{i}-q_{i}}_{2,\tau}+\norm{ q_{i}-\tilde{p}_{i}}_{2,\tau}\\
        &\leq \delta+\norm{ q_i \sum_{k=1}^{i-1}\tilde{p}_{k}}_{2,\tau}\\
        &\leq\delta+ \sum_{k=1}^{i-1}\norm{ q_{i}q_k\tilde{p}_{k}}_{2,\tau}\\
        &\leq (2n+1)\delta=\varepsilon
    \end{aligned}
    \end{equation}
    as desired.
\end{proof}

As a particular consequence of the above in combination with Lemma~\ref{lemma:areaprops}.\ref{lemma:areaprops_commute}, we can now show that commuting elements have small area, provided the trace is diffuse on $C(X)$. 

\begin{cor} \label{cor:commuting_smallarea}
    Let $A$ be a unital $C^*$-algebra containing a unital copy of $C(X)$ where $X$ is the Cantor space. If $\tau$ is a faithful, diffuse trace on $C(X)$, then any $a \in A \cap C(X)'$ has small area with respect to $(C(X),\tau)$.
\end{cor}
\begin{proof}
Pick $a\in A\cap C(X)'$ and $0<\varepsilon<1$. As $\tau$ is diffuse, there exists a partition of unity $\{p_i\}_{i=1}^n$ for $L^\infty(X,\mu_\tau)$ such that $\tau(p_i) = 1/n$ for all $i$ and $n\in\mathbb N$ such that
\begin{equation}
    \frac{4}{n^2} \leq \frac{\varepsilon}{n+1}\leq 1.
\end{equation}
Use Lemma \ref{lemma:approximating-multiple-projections} to find mutually orthogonal projections $\tilde p_1, ..., \tilde p_n$ with
\begin{equation}
    \|p_i-\tilde p_i\|_\tau \leq \varepsilon_0 := \frac{1}{n}\sqrt{\frac{\varepsilon}{n+1}} \leq \frac{1}{n}.
\end{equation}
Writing $\tilde p_{n+1} := 1- \sum_{i=1}^n \tilde p_i \in C(X)$, it follows from Kadison's inequality that
\begin{equation}
    \tau(\tilde p_{n+1}) \leq 1-\sum_{i=1}^n(1/n-\varepsilon_0) = n\varepsilon_0.
\end{equation}
Hence, $\{\tilde p_i\}_{i=1}^{n+1}$ is a partition of unity for $C(X)$ satisfying (using the same estimate as above for $\tau(\tilde p_i)$)
\begin{equation}
\begin{aligned}
    \max_{1\leq i \leq n+1} \tau(\tilde p_i)^2
        &\leq \max \left\{ (1/n+\varepsilon_0)^2, n^2\varepsilon_0^2 \right\} \\
        &\leq \max \left\{ 4/n^2, \varepsilon/(n+1) \right\} = \varepsilon/(n+1).
\end{aligned}
\end{equation}
As $a$ commutes with $C(X)$, we can use Lemma~\ref{lemma:areaprops}.\ref{lemma:areaprops_commute} to conclude that ${\mathcal{I}(a \mid \{\tilde p_i\}_{i=1}^{n+1})} \leq \varepsilon$, as desired.
\end{proof}

The next lemma lets us refine any partition of unity into one in which the projections have approximately equal traces.

\begin{lemma}
\label{lemma:equalizing-projections}
    Let $X$ be the Cantor space and $\tau$ a faithful, diffuse trace on $C(X)$. Let $\{p_{i}\}_{i=1}^{n}$ be a partition of unity of $C(X)$.
    For any $\varepsilon>0$ and $N\in \mathbb{N}$, there is a refinement $\{p_{i,j}\}_{1\leq i\leq n, 1\leq j\leq n_{i}}$ of $\{p_{i}\}_{i=1}^{n}$ such that $\sum_{i,j}p_{i,j}=1$, $K=\sum_{i=1}^{n}n_{i}\geq N$ and
    \begin{equation}
       1-K\min_{i,j}\tau(p_{i,j})\leq\varepsilon.
    \end{equation}
\end{lemma}
\begin{proof}
    We prove an analog in $L^{\infty}(X,\mu_\tau)$ and transport it approximately to $C(X)$ using Lemma~\ref{lemma:approximating-multiple-projections}.
    Let $t_{i}:=\tau(p_{i})$ and find $s_{i}=a_{i}/b_{i}$ with $a_i,b_i\in \mathbb{Z} $ such that $0\leq t_{i}-s_{i}\leq\frac{\varepsilon}{2n}$.
    Then $\sum_{i=1}^n s_i\geq 1-\varepsilon/2$.
    Let $m:=Nb_{1}\cdots b_{n}$ and $n_i := ms_i\in \mathbb{Z}$. 
    
    For each $i$ and $1\leq j\leq n_{i}$, use diffuseness of $\tau$ to find pairwise orthogonal projections $\tilde{p}_{i,j}\leq p_{i}$ in $L^{\infty}(X,\mu_\tau)$ such that $\tau(\tilde{p}_{i,j})=1/m$. 
    By defining $K:=\sum_{i=1}^{n}n_{i}\geq N\sum_{i=1}^n a_i\geq N$, we then have
    \begin{equation}
    \label{eq:equalizing-projections-estimate}
        K\tau(\tilde{p}_{i,j})=\sum_{i=1}^n n_i/m=\sum_{i=1}^n s_i\geq1-\varepsilon/2.
    \end{equation}
    Using Lemma~\ref{lemma:approximating-multiple-projections}, we find pairwise orthogonal projections $q_{i,j}$ in $C(X)$ such that 
    \begin{equation}
        \norm{ q_{i,j}-\tilde{p}_{i,j}}_{2,\tau}\leq \frac{\varepsilon}{2K}.
    \end{equation}
    For each $i,j$, the projection
    \begin{equation}
        p_{i,j}:=q_{i,j}p_{i}
    \end{equation}
    still satisfies $\norm{p_{i,j}-\tilde{p}_{i,j}}_{2,\tau}\leq \frac{\varepsilon}{2K}$ as $\tilde{p}_{i,j}\leq p_i$, so
    \begin{equation}
        \abs{\tau(p_{i,j})-\tau(\tilde{p}_{i,j})}\leq \norm{ p_{i,j}-\tilde{p}_{i,j}}_{2,\tau}\leq \frac{\varepsilon}{2K}
    \end{equation}
    by Kadison's inequality. Hence, \eqref{eq:equalizing-projections-estimate} gives
    \begin{equation}
        \begin{aligned}
            1-K\min_{i,j}\tau(p_{i,j})&\leq 1-K\min_{i,j}\tau(\widetilde p_{i,j}) +K\frac{\varepsilon}{2K}\\
            &\leq \frac{\varepsilon}{2}+\frac{\varepsilon}{2}=\varepsilon.
        \end{aligned}
    \end{equation}

    Replacing $p_{i,1}$ by $p_{i}-\sum_{j=2}^{n_{i}}p_{i,j}$, the value of $1-K\min_{i,j}\tau(p_{i,j})$ can only decrease and the $p_{i,j}$ now form a partition of unity, giving the desired refinement.
\end{proof}

\subsection{Obtaining single generation}
In this section, we prove our first main theorem. Its proof consists of three steps. The first and most involved step (Claim \ref{claim:inductive}) is an inductive argument producing a sequence of matrix algebras with respect to which the given sequence of generators has sufficiently small area.
The inductive step is Lemma \ref{lemma:inductive-step}; we include it as a separate lemma because it singles out the precise part for which we need strict/dynamical comparison, and the argument itself is quite technical. In the second step, we apply Shen's argument (Theorem \ref{thm:shen}) to obtain a sequence of projections which, together with the matrix algebras, recover the generators. The final step (Claim \ref{claim:singlegen}) then puts everything together to obtain single generation.

\begin{lemma}
    \label{lemma:inductive-step}
    Let $B$ be a unital, simple $C^*$-algebra with strict comparison and a unique trace $\tau$. Let $X$ be the Cantor space and suppose that $B$ contains $C(X)$ as a unital $C^*$-subalgebra and that $\tau$ is diffuse on $C(X)$.
    Suppose further that there exist a set $\{e^{(0)}_{ii}\}_{i=1}^{n_{0}}$ of pairwise orthogonal projections in $B$ of equal trace with $e^{(0)}_{11}\in C(X)$, and a partition of unity $\{\widetilde e^{(0)}_i\}_{i=1}^{n_{0}}$ for $C(X)$ such that $e^{(0)}_{ii}\leq \widetilde e^{(0)}_{i}$ for all $i$. Let $\varepsilon_{0}\geq 4/n_{0}$.
    
    Then for every $0<\varepsilon_1<\tau(e^{(0)}_{11})/2$ and partition of unity $\set{q_k}_{k=1}^N$ for $C(X)$, there exist $n_1\in \mathbb{N}$, a copy of a matrix algebra $M_{n_1}(\mathbb{C})$ in $B$ generated by matrix units $\{e^{(1)}_{ij}\}_{i,j=1}^{n_1}$, a partition of unity $\{\widetilde e^{(1)}_i\}_{i=1}^{n_{1}}$ for $C(X)$ refining $\{q_k\}_k$ and a partial isometry $w_1\in B$ such that, writing
    \begin{equation}
        p_0 := \sum_{i=1}^{n_0} e^{(0)}_{ii}, \qquad
        p_1 := \sum_{i=1}^{n_1} e^{(1)}_{ii}, \qquad
        v_1 := \sum_{i=1}^{n_1-1} \bigl(e^{(1)}_{i,i+1} + e^{(1)}_{i+1,i}\bigr),
    \end{equation}
    the following hold:
    \begin{enumerate}[label=(P\arabic*)]
    \item $e_{11}^{(1)}\in C(X)$ and $e_{11}^{(1)}\leq e_{11}^{(0)}$,
    \item $e_{ii}^{(1)}\perp e_{11}^{(0)}$ or $e_{ii}^{(1)}\leq e_{11}^{(0)}$ for each $i$,
    \item $e^{(1)}_{ii}\leq \widetilde e^{(1)}_{i}$ for each $i$,
    \item $w_{1}^{*}w_{1}=1-p_{1}$ and $w_1w_1^{*}\leq p_1$,
    \item $4/n_1 \leq \varepsilon_{1}$,
    \item $\tau(1-p_1)\leq \varepsilon_1$,
    \item $\tau(\mathcal{S}(w_{1}w_{1}^*\mid \{e_{ii}^{(1)}\}))\leq \varepsilon_{1}$,
    \item $\mathcal{I}(p_{0}ap_{0}\mid \{e_{ii}^{(0)}\})\leq \varepsilon_{0}$ for $a\in\{v_1, w_1\}$.
    \end{enumerate}
    If instead of $B$ having strict comparison, we suppose that $C(X)$ has dynamical comparison relative to $B$ and that $\{e^{(0)}_{ii}\}_{i=1}^{n_{0}}$ is a subset of $C(X)$, then the same conclusion holds and, moreover, $\{e_{ii}^{(1)}\}_{i=1}^{n_1}\subset C(X)$.
\end{lemma}

Before continuing with the proof, we first make a couple of remarks clarifying the data constructed in the lemma. The main output is the copy of a matrix algebra $M_{n_1}(\mathbb C) = \Span\{ e^{(1)}_{ij}\}_{i,j=1}^{n_1}$ for which the diagonal matrix elements $e^{(1)}_{ii}$ behave nicely (conditions (P1) and (P2)) with respect to the input element $e^{(0)}_{11}$ and refine a given partition of unity $\{q_k\}_k$ up to a defect projection $1-p_1$ (as a consequence of condition (P3)). This defect projection has small trace (conditon (P6)) and is Murray-von Neumann subequivalent to $p_1$ (condition (P4)). The partial isometry $w_1$ implementing this subequivalence and the sum $v_1 = \sum_{i=1}^{n_1-1} (e^{(1)}_{i,i+1}+ e^{(1)}_{i+1,i})$ of off-diagonal matrix units have small area with respect to the input set of projections $\{e^{(0)}_{ii}\}_{i=1}^{n_0}$ (condition (P8)). Finally, we have a partition of unity $\{\widetilde e^{(1)}_i\}_{i=1}^{n_{1}}$ for $C(X)$ such that each $e^{(1)}_{ii}$ is below $\widetilde e^{(1)}_i$ (condition (P3)). This partition will only be used to construct the next set of output data in the inductive sequence. The input partition of unity $\{q_k\}$ will be one with respect to which a given generator has small area.

\begin{proof}[Proof of Lemma \ref{lemma:inductive-step}]
Let $n_0$, $\varepsilon_0$, $\{e^{(0)}_{ii}\}_{i=1}^{n_0}$, $\{\widetilde e^{(0)}_{i}\}_{i=1}^{n_0}$ be as in the statement.
Fix $\varepsilon_1>0$ and let $\set{q_k}_{k=1}^N$ be a partition of unity for $C(X)$.

Note that the set of projections $\{e^{(0)}_{11}q_{j}\}_{j} \cup \{(1-e^{(0)}_{11})\widetilde e^{(0)}_iq_{j}\}_{i,j}$ is a partition of unity for $C(X)$.
Lemma \ref{lemma:equalizing-projections} provides $n_1\geq 4/\varepsilon_1$ -- giving condition (P5) -- and a partition of unity $\widetilde e^{(1)}_{1},\dots,\widetilde e^{(1)}_{n_1}$ for $C(X)$ refining $\{e_{11}^{(0)}q_{j}\}_{j}\cup\{(1-e_{11}^{(0)})\widetilde e^{(0)}_iq_{j}\}_{i,j}$ such that
\begin{equation} \label{eq:fihavesimilartrace}
    1-n_1\min_{i}\tau(\widetilde e^{(1)}_{i})\leq \frac{\varepsilon_1}{4}.
\end{equation}
By reindexing if necessary, we may assume that $(\widetilde e^{(1)}_i)_i$ is subordinate to $(\widetilde e^{(0)}_{i})_i$ and $\widetilde e^{(1)}_{1} \leq e^{(0)}_{11}$.
Using zero-dimensionality of $X$, we find a projection $e^{(1)}_{11}\in C(X)$ such that $e^{(1)}_{11}\leq \widetilde e^{(1)}_{1}$ and 
\begin{equation} \label{eq:f11traceclosetofi}
    \min_{i}\tau(\widetilde e^{(1)}_{i})-\frac{\varepsilon_1}{4n_1}\leq \tau(e^{(1)}_{11})<\min_{i}\tau(\widetilde e^{(1)}_{i}).
\end{equation}
This ensures (P1).
Applying strict comparison or dynamical comparison gives for each $1\leq i\leq n_1$ partial isometries $f_{i}\in B$ for such that $f_{i}^{*}f_{i}=e^{(1)}_{11}\in C(X)$ and $f_{i}f_{i}^{*}\leq \widetilde e^{(1)}_{i}$.
In the case of dynamical comparison, we furthermore have $f_{i}f_{i}^{*}\in C(X)$ for all $i$.
For each $i$ and $j$, write $e^{(1)}_{ij} := f_{i}f_{j}^{*}$. 
Then $\Span(e^{(1)}_{ij})$ is a copy of $M_{n_1}(\mathbb C)$ with unit $p_1:=\sum e^{(1)}_{ii}$ such that each $e^{(1)}_{ii}$ is under $\widetilde e^{(1)}_i$, so (P3) holds. As the $\widetilde e^{(1)}_i$ are a refinement of $\{e^{(0)}_{11}, 1-e^{(0)}_{11}\}$, we have also arranged (P2).

To prove that condition (P8) for $a=v_1$ is satisfied, observe the following. Since $e^{(1)}_{ii}\leq \widetilde e^{(1)}_i$ and $(\widetilde e^{(1)}_i)_i$ is subordinate to $(\widetilde e^{(0)}_i)_i$, for each $1\leq k\leq n_1$ there is a unique $1\leq i_k\leq n_{0}$ such that $e^{(1)}_{kk}\leq \widetilde e^{(0)}_{i_k}$, and $i_{k+1} \in \{i_k, i_k+1\}$. Consequently, if for arbitrary $i$, $j$ and $k$ we have
\begin{equation}
    0 \neq e^{(0)}_{ii} e^{(1)}_{k,k+1} e^{(0)}_{jj} = e^{(0)}_{ii} \widetilde e^{(0)}_{i_k} e^{(1)}_{k,k+1} \widetilde e^{(0)}_{i_{k+1}} e^{(0)}_{jj},
\end{equation}
then also $0\neq e^{(0)}_{ii} \widetilde e^{(0)}_{i_k} \leq \widetilde e^{(0)}_{i} \widetilde e^{(0)}_{i_k}$ and $0 \neq \widetilde e^{(0)}_{i_{k+1}}e^{(0)}_{jj} \leq \widetilde e^{(0)}_{i_{k+1}}\widetilde e^{(0)}_{j}$. 
This is only possible for $i=i_k$ and $j=i_{k+1}\in \{i_k, i_k+1\}$, so $j=i$ or $j=i+1$. 
Therefore, the element $v_{1}=\sum_{i=1}^{n_{1}-1}e_{i,i+1}^{(1)}+e_{i+1,i}^{(1)}$ satisfies 
\begin{equation}
        \mathcal{I}(p_{0}v_{1}p_{0}\mid \{e_{ii}^{(0)}\})\leq 4n_{0}\tau(e_{11}^{(0)})^{2}.
\end{equation}
Since $\{e_{ii}^{(0)}\}_{i=1}^{n_0}$ are pairwise orthogonal and of equal trace, we have $\tau(e^{(0)}_{ii}) = \tau(e_{11}^{(0)})\leq 1/n_{0}$ for all $i$, so we can further bound the above expression as
\begin{equation}
        \mathcal{I}(p_{0}v_{1}p_{0}\mid \{e_{ii}^{(0)}\})\leq \frac{4}{n_{0}}\leq \varepsilon_{0},
\end{equation}
proving that (P8) holds.
        
To show (P6), we compute
\begin{equation}
\label{eq:trace-estimate-1}
\begin{aligned}
    \tau(1-p_1)&=1-\sum_{i=1}^{n_1}\tau(e^{(1)}_{ii})\\
    &=1 -n_1\tau(\widetilde e^{(1)}_{1}) + n_1(\tau(\widetilde e^{(1)}_{1})-\tau(e^{(1)}_{11}))\\
    &\leq \frac{\varepsilon_1}{4}+ n_1\frac{\varepsilon_1}{4n_1}=\frac{\varepsilon_1}{2}\,
\end{aligned}
\end{equation}
using \eqref{eq:fihavesimilartrace} and \eqref{eq:f11traceclosetofi}.
    
We will now construct the partial isometry $w_1$ such that conditions (P4), (P7), and (P8) for $a=w_1$ hold. 
By combining \eqref{eq:fihavesimilartrace} and \eqref{eq:f11traceclosetofi}, we find
\begin{equation}
\label{eq:lower-bound}
\begin{aligned}
    \tau(e^{(1)}_{11})
        &\geq \min_{i}\tau(\widetilde e^{(1)}_{i})-\frac{\varepsilon_1}{4n_1}\\
        &\geq \frac{1}{n_1}\left(1-\frac{\varepsilon_1}{4}\right) - \frac{\varepsilon_1}{4n_1}\\
        &=\frac{1}{n_1}\left(1-\frac{\varepsilon_1}{2}\right) > \frac{\varepsilon_1}{2n_1}.
\end{aligned}
\end{equation}
By recalling that the $e^{(1)}_{ii}$ are pairwise orthogonal, we also have
\begin{equation}
\label{eq:upper-bound}
    \tau(e^{(1)}_{11}) \leq \frac{1}{n_1} \leq \frac{\varepsilon_1}{2}.
\end{equation}
This implies that there exists an $m\in \mathbb{N}$ such that $\varepsilon_1/2 <m\tau(e^{(1)}_{11})\leq\varepsilon_1$ and, by \eqref{eq:lower-bound}, it satisfies $m\leq n_1$. As such, we have
\begin{equation}
\label{eq:trace-estimate-2}
    \tau\left(\sum_{i=1}^m e^{(1)}_{ii} \right) =m\tau(e^{(1)}_{11}) >\frac{\varepsilon_1}{2}.
\end{equation}
When working with strict comparison, combining \eqref{eq:trace-estimate-1} and \eqref{eq:trace-estimate-2} immediately gives a $w_1\in B$ satisfying (P4).
In the setting of dynamical comparison, note that $1-p_1$ and $\sum_{j=1}^m e_{ii}^{(1)}$ are both in $C(X)$, so we can use the same trace estimates to obtain a $w_1 \in B$ satisfying (P4).
Moreover, we have
\begin{equation}
\begin{aligned}
    \tau(\mathcal{S}(w_1w_1^*\mid \{e^{(1)}_{ii}\}))
        &= \tau\left(\mathcal{S}\left(\sum_{i=1}^m e^{(1)}_{ii}\biggm| \{e^{(1)}_{ii}\}\right)\right) \\
        &= m \tau(e^{(1)}_{11}) \leq \varepsilon_1,
\end{aligned}
\end{equation}
so condition (P7) holds.

Finally, to show (P8) for $a=w_1$, we claim that $e^{(1)}_{ii}\leq e^{(0)}_{11}$ for all ${1\leq i\leq m}$. To this, denote by $M$ the number of $e^{(1)}_{ii}$ which are under $e^{(0)}_{11}$. Recalling that $(\widetilde e^{(1)}_i)_i$ is subordinate to $(\widetilde e^{(0)}_i)_i$, the tuple $(e^{(1)}_{ii})_i$ is subordinate to $(\widetilde e^{(0)}_{ii})_i$, so $e^{(1)}_{ii} \leq e^{(0)}_{11}$ for all $1\leq i \leq M$. Using condition (P2) (which we have already shown) and the fact that $\{e^{(1)}_{ii}\}_i\cup\{1-p_1\}$ partitions $B$, it follows that
\begin{equation}
    e^{(0)}_{11} \leq \sum_{i=1}^M e^{(1)}_{ii} + 1-p_1,
\end{equation}
so
\begin{equation}
    M\tau(e^{(1)}_{11}) + \tau(1-p_1) \geq \tau(e^{(0)}_{11}) \geq 2\varepsilon_1.
\end{equation}
In particular, this gives
\begin{equation}
    M\tau(e^{(1)}_{11}) \geq \varepsilon_1 \geq m \tau(e^{(1)}_{11}),
\end{equation}
so indeed $e^{(1)}_{ii}\leq e^{(0)}_{11}$ for all $1\leq i\leq m$.
Subsequently, the expression
\begin{equation}
    e^{(0)}_{jj}p_{0}w_1p_{0}=e^{(0)}_{jj}w_1p_{0}
\end{equation}
is only non-zero for $j=1$, so 
\begin{equation}
    \mathcal{I}(p_{0}w_1p_{0}\mid \{e^{(0)}_{ii}\})\leq \frac{1}{n_{0}} \leq \varepsilon_{0},
\end{equation}
which settles (P8) for $a=w_1$, finishing the proof.
\end{proof}

\begin{thm}
    \label{thm:singlegeneration}
    Let $C(X) \subseteq B \subseteq A$ be unital inclusions of $C^*$-algebras, where $X$ is the Cantor space. Suppose that $A$ is separable and that $B$ is simple with a unique trace $\tau$, and that $\tau$ is diffuse on $C(X)$.
    Suppose further that $B$ has strict comparison or that $C(X)$ has dynamical comparison relative to $B$.
    If there is a countable set $\set{x_n}_n\subset A$ of elements generating $A$ which have small area with respect to $(C(X),\tau)$, then $A$ is singly generated.
\end{thm}

\begin{proof}
We first prove the following claim:
\begin{claim} \label{claim:inductive}
    For each $l\in \mathbb{N}$, there is a matrix algebra $M_{n_l}(\mathbb C) = \Span\{e^{(l)}_{ij}\}_{i,j=1}^{n_l}$ in $B$ with unit $p_l := \sum_{i=1}^{n_l} e^{(l)}_{ii}$, a partition of unity $\{\widetilde e^{(l)}_{i}\}_{i=1}^{n_{l}}$ and a partial isometry $w_l \in B$ such that conditions (P1)-(P8) in Lemma \ref{lemma:inductive-step} hold (replacing the index $0$ by $l-1$ and $1$ by $l$), as well as
    \begin{enumerate}
        \item[(P9)] $\mathcal{I}(p_{l}x_lp_{l}\mid \{e_{ii}^{(l)}\})\leq \varepsilon_{l}$,
        \item[(P10)] $\sqrt{12\varepsilon_l} + \sqrt{\varepsilon_l}/2 \leq \tau(e^{(l-1)}_{11})/3$.
    \end{enumerate}
    If $C(X)$ has dynamical comparison relative to $B$, then the projection $e^{(l)}_{ii}$ is in $C(X)$ for all $l\in \mathbb{N}$ and $1\leq i \leq n_l$.
\end{claim}
\begin{proof}[Proof of claim]
We run an inductive argument using Lemma~\ref{lemma:inductive-step}. To start the inductive process, it suffices to define the index $0$ input data for Lemma~\ref{lemma:inductive-step}; the base case $l=1$ (as well as the other cases) then follows from the inductive step. 
Setting $n_0 := 1$, $e^{(0)}_{11} := \widetilde e_1^{(0)}:=1 \in C(X)$ and $\varepsilon_0 := 4$ gives the desired input data.

For the inductive step, suppose we are given input data\footnote{Note here that the output of Lemma~\ref{lemma:inductive-step} also satisfies the conditions of the input data.} as in Lemma~\ref{lemma:inductive-step} for $l-1$. Fix $\varepsilon_l>0$ such that
\begin{equation}
    \varepsilon_l \leq \frac{\tau(e^{(l-1)}_{11})^2}{9(2\sqrt{12} + 1/2)^2}.
\end{equation}
This arranges (P10). By assumption, there exists a partition $\{q_k\}_{k=1}^N$ of unity for $C(X)$ such that
\begin{equation}
\label{eq:xlsmallarea}
    \mathcal{I}(x_l \mid \{q_k\}_k) \leq \varepsilon_l.
\end{equation}
Applying Lemma~\ref{lemma:inductive-step} in $B$ for $\varepsilon_l$ (which, by definition, is bounded from above by $\tau(e^{(l-1)}_{11})/2$) and $\{q_k\}_k$, there exists a matrix algebra $M_{n_l}(\mathbb C) = \Span\{e^{(l)}_{ij}\}_{i,j=1}^{n_l}$ in $A$ with unit $p_l = \sum_{i=1}^{n_l} e^{(l)}_{ii}$, and a partial isometry $w_l \in A$ satisfying conditions (P1)-(P8). Moreover, note that $\{e_{ii}^{(l)}\}_i \cup\{(1-p_l)q_k\}_k$ is a refinement of $\{q_k\}_k$, by (P3) and the fact that $\{\widetilde e_{i}^{(l)}\}_i$ refines $\{q_k\}_k$. Combining \eqref{eq:xlsmallarea} and Lemma \ref{lemma:areaprops}.\ref{lemma:areaprops_refine}, it then follows that (P9) also holds. As such, all conditions are satisfied, finishing the inductive step. This proves the claim.
\end{proof}

Fix $l\in \mathbb N$, and consider the corner $p_{l}Ap_{l}$ of $A$, in which the copy of $M_{n_{l}}(\mathbb C)$ is now unital. We wish to apply Theorem~\ref{thm:shen} to the elements $p_{l}ap_{l}$, $w_{l}ap_{l}$, $p_{l}aw_{l}^{*}$, and $w_{l}aw_{l}^{*}$ for $a\in\{x_{l}, v_{l+1}, w_{l+1}\}$. For all $a\in \{x_l,v_{l+1},w_{l+1}\}$, it follows from (P8) and (P9) that 
\begin{equation}
    \mathcal{I}(p_{l}ap_{l})\mid \{e_{ii}^{(l)}\})\leq \varepsilon_l,
\end{equation}
and using (P7) and Lemma~\ref{lemma:areaprops}.\ref{lemma:areaprops_onesided}, we see that 
\begin{equation}
\begin{aligned}
    \mathcal{I}(w_{l}ap_{l}\mid \{e_{ii}^{(l)}\})&\leq \varepsilon_{l},\\
    \mathcal{I}(p_{l}aw_{l}^{*}\mid \{e_{ii}^{(l)}\})&\leq \varepsilon_{l},\\
    \mathcal{I}(w_{l}aw_{l}^{*}\mid \{e_{ii}^{(l)}\})&\leq \varepsilon_{l}\, .
\end{aligned}
\end{equation}
Note that $\tau$ is not necessarily a tracial state on $p_lAp_l$ whereas $\tau(\, \cdot \,)/\tau(p_l)$ is, so these estimates need to be divided by $\tau(p_l)^2$ to fit the context of Theorem~\ref{thm:shen}. The theorem then applies with $c= \sqrt{12\varepsilon_l}/\tau(p_l)$ and produces a projection $r_{l}\in p_{l}Ap_{l}$ such that 
\begin{equation}\label{eq:shenprojection}
    C^{*}(r_{l},M_{n_{l}}(\mathbb C))=C^{*}(p_{l}ap_{l},w_{l}ap_{l},p_{l}aw_{l}^{*},w_{l}aw_{l}^{*},M_{n_{l}}\mid a\in\{x_l, v_{l+1}, w_{l+1}\})
\end{equation}
and 
\begin{equation}
    \frac{\tau(\mathcal{S}(r_{l}\mid \{e_{ii}^{(l)}\}))}{\tau(p_l)}\leq 2c+2/n_l \leq \frac{\sqrt{12\varepsilon_l}}{\tau(p_l)} + \frac{\sqrt{\varepsilon_l}}{2}.
\end{equation}
The last inequality follows from (P5). Condition (P10) and the fact that $\tau(p_l)\leq 1$ allow us to rewrite this as
\begin{equation} \label{eq:support-of-r_lv4}
    \tau(\mathcal{S}(r_{l}\mid \{e_{ii}^{(l)}\}))\leq \frac{\tau(e_{11}^{(l-1)})}{3}.
\end{equation}
Let $N:=\tau(\mathcal{S}(r_{l}\mid \{e_{ii}^{(l)}\}))/\tau(e_{11}^{(l)})$ be the number of projections $\{e_{ii}^{(l)}\}$ that partition $\mathcal{S}(r_{l}\mid \{e_{ii}^{(l)}\})$.
We claim that there are at least $N$ projections $e_{ii}^{(l)}$ such that $e_{ii}^{(l)}\leq e_{11}^{(l-1)}-e_{11}^{(l)}$. To see this, recall from (P2) that for each $1\leq i\leq n_l$ we either have $e_{ii}^{(l)}\perp e_{11}^{(l-1)}$ or $e_{ii}^{(l)}\leq e_{11}^{(l-1)}$.
Since the projections $\{e_{ii}^{(l)}\}_{i=1}^{n_l}\cup\{1-p_l\}$ partition the unit of $A$, it suffices to show that
\begin{equation}
    \tau(e_{11}^{(l-1)}-e_{11}^{(l)})-\tau(1-p_l)\geq N\tau(e_{11}^{(l)}).
\end{equation}
This follows from (\ref{eq:support-of-r_lv4}), (P1) and (P5). 
By permuting the matrix units $e^{(l)}_{ij}$ appropriately, we may therefore assume that $\mathcal{S}(r_{l}\mid \{e_{ii}^{(l)}\})\leq e_{11}^{(l-1)}-e_{11}^{(l)}$. As $r_l \in p_l A p_l$, we have 
\begin{equation}
    r_l\leq \mathcal{S}(r_l\mid \{e_{ii}^{(l)}\})\leq e_{11}^{(l-1)}-e_{11}^{(l)},
\end{equation}
so $r_l$ commutes with $e_{11}^{(j)}$ for all $j\in\mathbb{N}$.
Note also that $r_l$ and $e_{11}^{(l)}$ commute with $w_1^*w_1=1-p_1$.
Indeed, by construction, $r_l\leq e_{11}^{(l-1)}\leq \dots \leq e_{11}^{(1)}\leq p_1$ for $l\geq 2$ and also $r_1\leq p_1$ by definition.

Hence, we obtain a sequence of projections $r_l \in p_lAp_l$ for all $l\in\mathbb N$, such that the set
\begin{equation}
    S = \{ e^{(l)}_{11}, r_l\}_{l\in\mathbb N} \cup \{1-p_1\}
\end{equation}
consists of pairwise commuting projections. Let 
\begin{equation}
    y=w_{1}+w_{1}^{*}+v_1 \in A
\end{equation}
and note that $y$ is self-adjoint.

\begin{claim} \label{claim:singlegen}
    $A=C^*(y, S)$.
\end{claim}
\begin{proof}[Proof of claim]
    The proof is a slight modification of \cite[Theorem~4.2]{Shen}.
    First notice that since $w_1^*w_1\in  S$,
    \begin{equation}
        w_1^*=w_1^*w_1 y\in C^*(y,S)
    \end{equation}
    and so 
    \begin{equation}
        v_1 = y - w_1 - w_1^* \in C^*(y,S)
    \end{equation}
    as well.
    Since $e_{11}^{(1)}\in S$, it follows by a standard argument (see also the proof of \cite[Theorem~3.4]{Shen}) that 
    $M_{n_1}(\mathbb C)\subseteq C^*(y,S)$.
    Hence $C^*(r_1,w_1,M_{n_1}(\mathbb C))\subseteq C^*(y,S)$ and so 
    \begin{equation}
        x_1,v_2,w_2\in C^*(y,S)
    \end{equation}
    by (P4) and \eqref{eq:shenprojection}.

    Since $e_{11}^{(i)}\in S$, we can proceed inductively, repeatedly using (P4) and \eqref{eq:shenprojection}, to get $x_l\in C^{*}(y,S)$.
    Since $A=C^{*} (x_l\mid l\in \mathbb{N})$, this finishes the proof.
\end{proof}

Finally, Lemma~\ref{lemma:single-generation-with-projections} shows that $C^*(S)=C^*(z)$ for some self-adjoint $z\in C^*(S)\subseteq A$.
Hence, $A$ is generated by two self-adjoint elements $y$ and $z$ and is therefore generated by the element $y+iz$.
\end{proof}

\subsection{Arranging small area for normalizers}
In order to obtain concrete classes of $C^*$-algebras to apply Theorem~\ref{thm:singlegeneration} to, we will show in this section that any unital, separable $C^*$-algebra $A$ containing $C(X)$ with $X$ the Cantor space as a unital, regular $C^*$-subalgebra, has a countable set of generators with small area, provided $A$ admits a faithful trace $\tau$ which is diffuse on $C(X)$ (Corollary~\ref{cor:smallgenerators}). Note here that diffuseness is automatic if $C(X)$ is regular in a simple, monotracial $C^*$-subalgebra of $A$ as a consequence of Corollary~\ref{cor:diffuse}.

Our strategy consists of showing that any partial isometry $v\in N_A(C(X))$ has small area w.r.t.\ $(C(X),\tau)$ (Proposition~\ref{prop:cnpi_smallarea}), which suffices as a consequence of Lemma~\ref{lemma:clopen-normalizers}.
To do so, we pass to the von Neumann completion $\mathcal{M} := \pi_\tau(A)''$ of $A$ in the GNS representation $\pi_\tau$ associated to $\tau$ and extend $v$ to a unitary $u$ in $\mathcal{M}$ (see Lemma~\ref{lemma:npi_extends}).
The weak closure of $\pi_\tau(C(X))$ is identified with $L^{\infty}(X, \mu_{\tau})$, so that $\pi_{\tau}|_{C(X)}$ is the natural embedding of $C(X)$ into $L^{\infty}(X,\mu_\tau)$.
Since $\tau$ is diffuse, the induced measure is as well, which allows us to subdivide projections in $L^\infty(X,\mu_\tau)$ into arbitrarily many pieces.
Using the Rokhlin lemma in the form of Lemma~\ref{lemma:rokhlin-lemma} then produces a partition of $1_{\mathcal{M}}$ into orthogonal projections approximately compatible with the conjugation action of $u$.
We then lift the partition approximately to $C(X)$ and maintain approximate compatibility with conjugation by $v$ by using Lemma~\ref{lemma:approximating-multiple-projections}.
Finally, we verify that the area of $v$ is small with respect to these projections.

The first lemma is a straightforward consequence of the Rokhlin lemma. We provide details for convenience.

\begin{lemma}
\label{lemma:rokhlin-lemma}
    Let $(X,\mu)$ be a standard Borel probability space and let $T\in \Aut(X)$ be an invertible measure preserving transformation.
    Then for every $\varepsilon>0$ there exist $n\in \mathbb{N}$ and measurable sets $E_1,\dots, E_n$ such that $E_k,T(E_k),\dots, T^{k-1}(E_k)$ are disjoint for $1 \leq k \leq n$.
    Moreover, setting 
    \begin{equation}
        E_\varepsilon=X\setminus \bigcup_{\substack{1\leq i\leq n\\0\leq j\leq i-1}}T^j(E_i)
    \end{equation}
    we have $\mu(E_\varepsilon)<\varepsilon$  and
    $T^k(E_k)\subseteq E_k\cup E_\varepsilon$ for each $1\leq k\leq n$.
\end{lemma}

\begin{proof}
    For each $k\in \mathbb{N}$, let $X_k$ denote the set of points in $X$ with orbit of size exactly $k$ and $X_\infty$ the set of points with infinite orbit.
    Clearly, $T$ restricts to each $X_k$ for $1\leq k\leq\infty$.
    Pick $n$ large enough so that 
    \begin{equation}
        \mu\left(\bigcup_{n\leq k<\infty}X_k\right)<\varepsilon/2.
    \end{equation}

    For each $1\leq k< n$, find $E_k\subseteq X_k$ such that $X_k$ is the disjoint union of the orbit of $E_k$, see \cite[p.~70,~Lemma~1]{halmos56} for instance.
    Using the Rokhlin lemma, see \cite[p.~71,~Lemma~2]{halmos56} for example, find $E_n \subseteq X_\infty$ such that $E_n, T(E_n), \dots T^{n-1}(E_n)$ are disjoint and 
    \begin{equation}
        \mu\left(\bigcup_{0\leq j\leq n-1}T^j(E_n)\right)>\mu(X_\infty)-\varepsilon/2. 
    \end{equation}
    
    Notice that since $T$ is invertible, $T^k(E_k)$ is disjoint from
    \begin{equation}
        \bigcup_{1\leq j\leq k-1}T^j(E_k)
    \end{equation} 
    and from $X\setminus X_k$, so we must have 
    $T^k(E_k)\subseteq E_k\cup E_\varepsilon$,
    where 
    \begin{equation}
        E_\varepsilon=X\setminus \bigcup_{\substack{1\leq i\leq n\\0\leq j\leq i-1}}T^j(E_i).
    \end{equation}
    The measure estimate holds by construction.
\end{proof}

Lemma~\ref{lemma:lifting-projections-v2} serves as a topological version of Lemma \ref{lemma:rokhlin-lemma}.
If $v$ is invertible so that it is in fact a unitary, and furthermore if the homeomorphism induced by $\Ad_v$ has no periodic points, the analog has been proved in \cite{bdm05}. 
In general, the subspace of points with finite orbits need not be clopen in $X$ and so we cannot reduce to the aforementioned result by restriction.
Moreover, for the purposes of Theorem~\ref{thm:singlegeneration}, we are required to work with non-unitary partial isometries since there is no guarantee that the invertible normalizers of $C(X)$ generate $A$.
The output set of projections in the following lemma is indexed as follows: for each pair $(i,j)$, the set $\{p^{(i,j)}_k \}_{k\in \mathbb{Z}/i\mathbb{Z}}$ is the orbit of a projection $p^{(i,j)}_0$ with length $i$, up to a defect contained in the projection $p_\varepsilon$ (see property (iv)).

\begin{lemma}
    \label{lemma:lifting-projections-v2}
    Let $A$ be a unital $C^*$-algebra containing a unital copy of $C(X)$ with $X$ the Cantor space. Suppose $A$ has a faithful trace $\tau$ which is diffuse on $C(X)$.
    For any partial isometry $v\in N_A(C(X))$ and $\varepsilon>0$, there exist integers $n\in \mathbb N$ and $ \ell_1, ..., \ell_n \in \mathbb N$ and pairwise orthogonal projections\footnote{Note that these may be 0. By the index set of $k$ being $\mathbb{Z}/i\mathbb{Z}$, we mean that $k\in\{0,...,i-1\}$ and identify $k=i$ with $k=0$ for ease of notation.} 
    \begin{equation}
    \{p^{(i,j)}_k \mid i\in \{1,...,n\}, j\in \{1,...,\ell_i\}, k\in \mathbb{Z}/i\mathbb{Z}\}
    \end{equation} 
    and $p_{\varepsilon}$ in $C(X)$ such that
    \begin{enumerate}[label=(\roman*)]
        \item $p_\varepsilon + \sum_{i,j,k} p_{k}^{(i,j)}=1$;
        
        \item $\tau(p_\varepsilon)\leq \varepsilon$;

        \item $\tau(p_k^{(i,j)})^{2}\leq \frac{\varepsilon}{i \ell_i n}$ for all $1\leq i\leq n$ and all $j$ and $k$;
        
        \item $\Ad_{v} p^{(i,j)}_k \leq p^{(i,j)}_{k+1} + p_\varepsilon$ for all $i$, $j$ and $k$.
    \end{enumerate}
\end{lemma}

\begin{proof}
    Fix $\varepsilon>0$.
    Apply Lemma~\ref{lemma:rokhlin-lemma} with $\varepsilon/2$ to the automorphism $\Ad_u$ of $L^\infty(X,\mu_\tau)$, where $u\in N_{\mathcal{M}}(L^\infty(X,\mu_\tau))$ is a unitary extending $v$ as in Lemma~\ref{lemma:npi_extends}.
    We obtain (by considering indicator functions) $n\in\mathbb{N}$ and pairwise orthogonal projections $\widetilde p_\varepsilon$, $\widetilde p_k^{(i)}\in L^\infty(X,\mu_\tau)$ for $1\leq i\leq n$ and $0\leq k\leq i-1$.
    These satisfy the following:
    \begin{enumerate}[label=(\arabic*)]
        \item $\widetilde p_\varepsilon+\sum_{i,k} \widetilde p_k^{(i)}=1$,
        \item $\Ad_u \widetilde p_k^{(i)}=\widetilde p^{(i)}_{k+1}$ for all $i$ and $0\leq k\leq i-2$,
        \item $\Ad_u \widetilde p_{i-1}^{(i)}\leq \widetilde p_0^{(i)}+\widetilde p_\varepsilon$,
        \item $\tau(\widetilde p_\varepsilon)\leq \varepsilon/2$.
    \end{enumerate}

    Subdividing each $\widetilde p_0^{(i)}$ into $\widetilde p_0^{(i,j)}$ for $1\leq j\leq \ell_i$ of small enough trace (see (5'); this is possible by diffuseness of $\tau$) and setting 
    \begin{equation}
        \widetilde p_k^{(i,j)}:=\Ad_{u^k} \widetilde p_0^{(i,j)},
    \end{equation}
    we get pairwise orthogonal projections satisfying

    \begin{enumerate}[label=(\arabic*')]
        \item $\widetilde p_\varepsilon+\sum_{i,j,k} \widetilde p_k^{(i,j)}=1$,
        \item $\Ad_u \widetilde p_k^{(i,j)}=\widetilde p^{(i,j)}_{k+1}$ for all $i$ and $0\leq k\leq i-2$,
        \item $\Ad_u \widetilde p_{i-1}^{(i,j)}\leq \widetilde p_0^{(i,j)}+\widetilde p_\varepsilon$,
        \item $\tau(\widetilde p_\varepsilon)\leq \varepsilon/2$,

        \item $\tau(\widetilde p_k^{(i,j)})^{2}\leq \frac{\varepsilon}{2i \ell_i n}$ for all $i$, $j$ and $k$.
    \end{enumerate}
    Since $v=uv^*v$, items (2') and (3') also hold when $u$ is replaced by $v$.
    
    Fix $\delta>0$, and use Lemma~\ref{lemma:approximating-multiple-projections} to produce pairwise orthogonal projections $\{q^{(i,j)}_k\}_{i,j,k}$ in $C(X)$ such that for all $i$, $j$ and $k$, we have
    \begin{equation}
        \norm{ q_k^{(i,j)}-\widetilde p_k^{(i,j)}}_{2,\tau}\leq \delta.
    \end{equation}
    Now, these projections are no longer guaranteed to satisfy conditions (i)-(iv) in the statement, so this must be corrected. To this end, first let
    \begin{equation}
    \begin{aligned}
        p_0^{(1,1)}&:= q_0^{(1,1)},\\
        p_{0,\delta}^{(1,1)}&:=(1-p^{(1,1)}_0)\Ad_v p_0^{(1,1)},        
    \end{aligned}
    \end{equation}
    so $\Ad_v p_0^{(1,1)}\leq p_0^{(1,1)}+p_{0,\delta}^{(1,1)}$.
    We now wish to construct projections $p_0^{(1,2)}$ and $p_{0,\delta}^{(1,2)}$ such that $\Ad_v p_0^{(1,2)}\leq p_0^{(1,2)}+p_{0,\delta}^{(1,2)}$, while maintaining orthogonality.
    In particular, we need to ensure that $\Ad_v p_0^{(1,2)}$ does not intersect $p_0^{(1,1)}+p_{0,\delta}^{(1,1)}$.
    To do so, we cut by the complement of $\Ad_{v^*}(p_0^{(1,1)}+p_{0,\delta}^{(1,1)})$:
    \begin{equation}
    \begin{aligned}
        p_0^{(1,2)}&:= q_0^{(1,2)}(1-\Ad_{v^*}(p_0^{(1,1)}+p_{0,\delta}^{(1,1)}))\\
        p_{0,\delta}^{(1,2)}&:=(1-p_0^{(1,2)})\Ad_v p_0^{(1,2)}.
    \end{aligned}
    \end{equation}

    More generally, we employ the same strategy to inductively construct pairwise orthogonal projections $p_k^{(i,j)}$ and $p_{k,\delta}^{(i,j)}$ for $1\leq i\leq n$, $1\leq j \leq \ell_i$ and $k\in \mathbb{Z}/i\mathbb{Z}$ such that
    \begin{equation}
    \label{eq:pij-ad}
        \Ad_v p_k^{(i,j)}\leq p_{k+1}^{(i,j)}+p_{k,\delta}^{(i,j)}.
    \end{equation}
    To this end, suppose all such $p_k^{(i_0,j_0)}$ and $p_{k,\delta}^{(i_0,j_0)}$ ($k\in\mathbb{Z}/i_0\mathbb{Z}$) have been defined for all pairs $(i_0,j_0)$ up to but not including some $i\leq n$ and $j \leq \ell_i$ (using the lexicographic order on this index set). Let $p$ be the sum of all projections defined up to this point.
    For each $k \in \mathbb{Z}/i\mathbb{Z}$, let
    \begin{equation}
    \label{eq:pij-definition}
        p_k^{(i,j)}:=q_k^{(i,j)}(1-p)(1-\Ad_{v^*}p)\prod_{k'\neq k +1}(1-\Ad_{v^*} q_{k'}^{(i,j)}).
    \end{equation}
    Once this has been done, also define for each $k\in\mathbb{Z}/i\mathbb{Z}$
    \begin{equation}
    \label{eq:pij-eps-definition}
                p_{k,\delta}^{(i,j)} := (1-p^{(i,j)}_{k+1}) \Ad_v p^{(i,j)}_k.
    \end{equation}
    The inequality \eqref{eq:pij-ad} is then satisfied by construction.
    Note that each $\Ad_v p_{k}^{(i,j)}$ is orthogonal to $p$ and $q_{k'}^{(i,j)}$ for all $k'\neq k+1$, so it is orthogonal to all $p_{k}^{(i,j)}$ defined up to this point except possibly $p_{k+1}^{(i,j)}$.
    Therefore, $p_{k,\delta}^{(i,j)}$ is orthogonal to all $p_{k}^{(i_0,j_0)}$ defined up to this point.
    Moreover, since the $q_k^{(i,j)}$ are pairwise orthogonal, so are the $p_k^{(i,j)}$. The same then holds for their images under $\Ad_v$, so the $p_{k,\delta}^{(i,j)}$ are also pairwise orthogonal.
    
    Now that we have defined all $p_k^{(i,j)}$ and $p_{k,\delta}^{(i,j)}$, set
    \begin{equation}
        p_\varepsilon:=1-\sum_{i,j,k} p_k^{(i,j)},
    \end{equation}
    and note that by orthogonality, we have $p_{k,\delta}^{(i,j)}\leq p_\varepsilon$ for all $i$, $j$ and $k$.
    We have thus established conditions (i) and (iv) in the statement, so it remains to show that (ii) and (iii) hold provided $\delta$ is chosen small enough.

    We claim that for each $i$, $j$ and $k$, there is a constant $C_{i,j,k}$, depending only on $n$ and the $\ell_i$, such that
    \begin{equation} \label{eq:pij_tracediff}
        \norm{p_k^{(i,j)}-\widetilde p_k^{(i,j)}}_{2,\tau}\leq C_{i,j,k}\delta.
    \end{equation}
    Indeed, using the fact that $\|q_k^{(i,j)}-\tilde p_k^{(i,j)}\|_{2,\tau}\leq \delta$ and that all elements we are working with have norm one, in any expression defining $p_k^{(i,j)}$ or $p_{k,\delta}^{(i,j)}$ one can (inductively) replace all occurrences of $q_k^{(i,j)}$ by $\tilde p_k^{(i,j)}$ at the cost of a finite number $C_{i,j,k}$ of errors $\delta$.
    After this replacement, the right hand sides of \eqref{eq:pij-definition} and \eqref{eq:pij-eps-definition} collapse to $\tilde p_k^{(i,j)}$ and 0, respectively. 
    This proves the claim.

    Letting $C:=\sum_{i,j,k}^n C_{i,j,k}$, it follows from \eqref{eq:pij_tracediff} and conditions (1') and (4') that 
    \begin{equation}
        \tau(p_\varepsilon) = \tau(\widetilde p_\varepsilon) + \sum_{i,j,k} \tau(\widetilde p_k^{(i,j)}- p_k^{(i,j)}) \leq \varepsilon/2 + C\delta.
    \end{equation}
    Choosing
    \begin{equation}
      \delta\leq \min_{1\leq i\leq n}\left\{\frac{\varepsilon}{4Ci\ell_i n},\frac{\varepsilon}{2C}\right\}
    \end{equation}
    then guarantees that $\tau(p_\varepsilon)\leq \varepsilon$ and thus condition (ii) holds. 
    Moreover, we compute for all $i$, $j$ and $k$
    \begin{equation}
    \begin{aligned}
        |\tau(p_k^{(i,j)})^2-\tau(\tilde p_k^{(i,j)})^2|
            &\leq 2|\tau(p_k^{(i,j)})-\tau(\widetilde p_k^{(i,j)})| \\
            &\leq 2\norm{p_k^{(i,j)}-\widetilde p_k^{(i,j)}}_{2,\tau}\\
            &\leq \frac{\varepsilon}{2i\ell_i n}.
        \end{aligned}
    \end{equation}
    Combining the above estimate with condition (5') verifies condition (iii) and finishes the proof.
\end{proof}

We now apply the previous lemma to obtain the desired partition of unity.

\begin{prop} \label{prop:cnpi_smallarea}
    Let $A$ be a unital $C^*$-algebra containing a unital copy of $C(X)$ with $X$ the Cantor space. Suppose $A$ admits a faithful trace $\tau$ which is diffuse on $C(X)$.
    Any normalizing partial isometry $v \in N_A(C(X))$ has small area with respect to $(C(X),\tau)$.
\end{prop}

\begin{proof}
    Fix $\varepsilon>0$. 
    Let $n,\ell_1, ..., \ell_n \in\mathbb N$, $\{p^{(i,j)}_k\mid 1\leq i \leq n, 1\leq j \leq \ell_i, k\in \mathbb{Z}/i\mathbb{Z}\}$ and $p_\varepsilon$ be the integers and projections in $C(X)$ obtained by applying Lemma~\ref{lemma:lifting-projections-v2} to $v$ and $\varepsilon/3$.
    Since $v$ is a normalizer of $C(X)$, one can observe for arbitrary $q\in \{p^{(i,j)}_k\}_{i,j,k}\cup \{p_\varepsilon\}$ that $qvp_k^{(i,j)}\neq 0$ if and only if $q\Ad_v p^{(i,j)}_k \neq 0$, which implies that $q=p_{k+1}^{(i,j)}$ or $q=p_\varepsilon$.
    This allows us to bound the area of $v$ as follows:
    \begin{equation}
    \begin{aligned}
    \mathcal{I}(v\mid \{p_k^{(i,j)}&\}_{i,j,k}\cup \{p_\varepsilon\})
        \\
        &=\tau(p_\varepsilon)\sum_{q:qvp_\varepsilon\neq 0}\tau(q)+\sum_{i,j,k}\tau(p_k^{(i,j)})\sum_{q:qvp_k^{(i,j)}\neq 0} \tau(q)\\
        &\leq\tau(p_\varepsilon)\cdot 1 + \sum_{i,j,k}\tau(p_k^{(i,j)}) \left(\tau(p_\varepsilon) + \tau(p_{k+1}^{(i,j)}) \right) \\
        &\leq 2\tau(p_\varepsilon) + \sum_{i,j,k} \frac{\varepsilon}{3i\ell_i n} \\
        &\leq 2\varepsilon/3 + \varepsilon/3 =\varepsilon.
    \end{aligned}
    \end{equation}
    This finishes the proof.
\end{proof}

\begin{cor} \label{cor:smallgenerators}
    Let $A$ be a unital, separable $C^*$-algebra with a faithful trace $\tau$, containing a unital copy of $C(X)$ with $X$ the Cantor space. If $C(X)$ is regular in $A$ and $\tau$ is diffuse on $C(X)$, then there exists a countable set $\set{x_n}_n \subset A$ generating $A$ consisting of elements with small area w.r.t.\ $(C(X),\tau)$.
\end{cor}
\begin{proof}
By separability of $A$, regularity of $C(X)$ and Lemma \ref{lemma:clopen-normalizers}, there exists a countable set $\{b_n\}_n \subset N_A^c(C(X))$ of clopen normalizers of $C(X)$ generating $A$. Fix $n\in \mathbb N$. By Lemma \ref{lemma:clopen-normalizers}, we can write $b_n = v_n |b_n|$ for some partial isometry $v_n \in N_A^c(C(X))$. As $|b_n| \in C(X)$, it commutes with $C(X)$, so it has small area w.r.t.\ $(C(X),\tau)$ by Corollary~\ref{cor:commuting_smallarea}. The previous proposition implies that $v_n$ also has small area, so we can conclude that the countable set
\begin{equation}
    \set{v_n}_n \cup \set{|b_n|}_n
\end{equation}
generates $A$ and consists of elements of small area in $C(X)$.
\end{proof}

\subsection{Examples}

With Corollary~\ref{cor:smallgenerators} in hand, we are now ready to use Theorem~\ref{thm:singlegeneration} to provide a positive answer to the generator problem for several classes of $C^*$-algebras. An obvious class of $C^*$-algebras which contain the continuous functions on the Cantor space as a regular subalgebra, are crossed products of actions on the Cantor space, resulting in Corollary \ref{cor:sg_crossedprod} below. Note that the diffuseness assumption is no longer necessary as it is automatic due to Corollary~\ref{cor:diffuse}. We also want to highlight that the result can be generalized straightforwardly to the setting of groupoids, but we state it in the language of crossed products because that is where our main concrete examples lie.

\begin{cor} \label{cor:sg_crossedprod}
    Let $G$ be a countable discrete group with a subgroup $F$, and let $\alpha\colon G\curvearrowright X$ be an action of $G$ on the Cantor space $X$ such that the reduced crossed product $C(X)\rtimes_{\alpha,r} G$ has a faithful trace. Suppose the restricted action $\beta :=\alpha|_F\colon F \curvearrowright X$ induces a monotracial, simple crossed product $C(X)\rtimes_{\beta,r} F$ such that $C(X)\rtimes_{\beta,r} F$ has strict comparison or $C(X)$ has dynamical comparison relative to $C(X)\rtimes_{\beta,r} F$.
    Then $C(X)\rtimes_{\alpha,r} G$ is singly generated.
\end{cor}
\begin{proof}
First note that Corollary~\ref{cor:smallgenerators} holds for $A:= C(X)\rtimes_{\alpha,r} G$ as $C(X)$ is regular in $B:= C(X)\rtimes_{\beta,r} F \subseteq A$ so $\tau$ is diffuse on $C(X)$ by Corollary~\ref{cor:diffuse}. The statement now follows immediately from Theorem~\ref{thm:singlegeneration}.
\end{proof}

We want to highlight the following classes of $C^*$-algebras for which the above corollary applies with $F=G$, and therefore are singly generated. 
\begin{enumerate}
    \item Crossed products of minimal, uniquely ergodic actions with dynamical comparison\footnote{Until very recently, it was not known whether every minimal action on the Cantor space has dynamical comparison. The first examples without dynamical comparison have been obtained in \cite{paolo}.} (see remark below Definition~\ref{def:dynamicalcomp}) of countable, discrete, $C^*$-simple groups on the Cantor space. These crossed products are simple by \cite[Theorem~1.8]{BKKO17} and have a unique trace by \cite[Corollary~1.4]{bryderkennedy}.
    
    \item \label{sg_selfless} Selfless crossed products of actions on the Cantor space, see \cite{bell-selfless}.
    For instance, actions of groups $G$ with property PHP on Cantor space $X$ preserving a Borel probability measure such that some infinite amenable subgroup acts minimally and uniquely ergodically on $X$.
    
    \item \label{sg_profinite} Crossed products of free and minimal profinite actions (as defined in Remark \ref{rmk:profinite} below) of countable, discrete groups. Minimal profinite actions have dynamical comparison and are uniquely ergodic, see \cite[Lemma~2.4, Lemma~3.1]{joseph}.
    The crossed product is thus simple (\cite[Corollary on p.~122]{archboldspielberg}) and has a unique trace (\cite{zeller-meier}). 
    
    \item \label{sg_buncededdens} As a special case of the above, the non-nuclear Bunce--Deddens algebras recently introduced by Bell in \cite{bell}. 
    In this paper, it is shown that these crossed products arise as profinite actions on the Cantor space.
\end{enumerate}

Crossed products in class (\ref{sg_profinite}) above are not $\mathcal{Z}$-stable whenever the acting group is not inner amenable, i.e.\ when $G$ does not admit an atomless mean which is invariant under conjugation. This is a consequence of the fact that any tracial GNS completion of a $\mathcal{Z}$-stable $C^*$-algebra is McDuff whereas by \cite{choda}, the relative commutant of GNS completions of the crossed products in (3) are commutative whenever the acting group is not inner amenable, and therefore not McDuff (see also \cite[Proposition 4.1]{DV}).

\begin{rmk}[Profinite actions] \label{rmk:profinite}
    Let $\alpha\colon G\curvearrowright X$ be an action of a discrete countable group $G$ on the Cantor space $X$. The action $\alpha$ is said to be \emph{profinite} if it is induced by actions $\{\alpha_{i}\colon G\curvearrowright X_i\}_{i\in\mathbb{N}}$ on an inverse system of topological spaces
    \begin{equation}
        X\xrightarrow{} \dots\xrightarrow{} X_{3}\xrightarrow{\pi_{2}} X_{2}\xrightarrow{\pi_{1}} X_{1}
    \end{equation}
    with $X$ as its limit, and with surjective connecting maps.
    
    One can easily obtain free and minimal profinite actions of infinite, residually finite groups as follows.
    Given an infinite, residually finite group $G$, choose a decreasing sequence of normal subgroups 
    \begin{equation}
        G\trianglerighteq N_1 \trianglerighteq N_2\trianglerighteq\dots
    \end{equation}
    with trivial intersection.
    One obtains the inverse system 
    \begin{equation}
        X\rightarrow\dots \rightarrow G/N_2\rightarrow G/N_1
    \end{equation}
    equipped with the natural action of $G$ by coset shifting and with inverse limit homeomorphic to the Cantor space.
    The action of $G$ on $X$ is then free and minimal (see \cite[p.~3 and p.~4]{cortezpetite}).
\end{rmk}

The following corollary provides a family of amalgamated free products which are singly generated.

\begin{cor}
\label{cor:amalgams}
    Let $\{A_i\}_{i\in I}$ be an at most countable family of unital separable $C^*$-algebras with faithful traces, all containing a regular abelian subalgebra $C(X)$ with spectrum homeomorphic to the Cantor space.
    Suppose that $A_1$ is simple with a unique trace $\tau$ and that either $A_1$ has strict comparison or $C(X)$ has dynamical comparison w.r.t. $A_1$.
    If for each $i \in I$, there exists a faithful trace on $A_i$ which restricts to $\tau$ on $C(X)$, then the full amalgamated free product $A_1*_{C(X)}A_2*_{C(X)}\dots$ is singly generated.
    Consequently, any completion of the algebraic amalgamated free product is singly generated.
\end{cor}

\begin{proof}
    Apply Theorem~\ref{thm:singlegeneration} with $B=A_1$ and Corollary~\ref{cor:smallgenerators} for each of the inclusions $C(X)\subseteq A_i$; note that the unique trace on $A_1$ is diffuse on $C(X)$ by Corollary~\ref{cor:diffuse}.
    Since the full amalgamated product surjects onto any other completion, the result follows.
\end{proof}

As commutation automatically ensures small area due to Corollary~\ref{cor:commuting_smallarea}, any tensor product can be added on to these examples preserving single generation. 

\begin{cor} \label{cor:sg_tensorprod}
    Let $A$ be a $C^*$-algebra satisfying the conditions of Theorem~\ref{thm:singlegeneration}. For any unital, separable $C^*$-algebra $D$ and any $C^*$-norm $\alpha$ on the algebraic tensor product $A \odot D$, the $C^*$-algebra $A\otimes_\alpha D$ is singly generated.
\end{cor}
\begin{proof}
Let $X$, $A$ and $B$ be as in the statement of Theorem~\ref{thm:singlegeneration}; in particular, let $\tau$ be the trace on $B$, which is diffuse on $C(X)$, and let $\set{a_n}_n$ be a countable set of generators of $A$ with small area w.r.t.\ $(C(X),\tau)$. Fixing a countable set $\set{d_n}_n \subset D$ generating $D$, the elements $1_A\otimes d_n \in A\otimes_\alpha D$ ($n\in\mathbb N$) have small area w.r.t.\ $(C(X) \otimes 1_D,\tau)$ by Corollary~\ref{cor:commuting_smallarea}, as they commute with $C(X)\otimes 1_D$. Hence, we can apply Theorem \ref{thm:singlegeneration} for $C(X) \otimes 1_D \subseteq B \otimes 1_D \subseteq A\otimes_\alpha D$ and the countable set
\begin{equation}
    \set{a_n\otimes 1_D}_n \cup \set{1_A\otimes d_n}_n 
\end{equation}
to find that $A\otimes_\alpha D$ is singly generated.
\end{proof}

\section{Tracially AF algebras} \label{sec:TAF}

In this section, we solve the generator problem for simple, separable, tracially approximately finite-dimensional $C^*$-algebras. We begin by recalling their definition and collect the structural properties that will be used in the proof of our main theorem. We will make particular use of real rank zero and strict comparison, as well as of a strengthened form of the tracial approximations (Proposition~\ref{prop:esstracZstab}) which provides additional control over the well-behaved corner.

To simplify the notation in this section, we write $a \approx_\varepsilon b$ when $\|a-b\|< \varepsilon$ and $a \in_\varepsilon B$ when $\mathrm{dist}(a,B) < \varepsilon$.

\subsection{Properties of tracially AF algebras}
We first recall the definition of a tracially AF algebra. The non-unital version can be traced back to: \cite[Definition 3.1]{LinTRan}.

\begin{defn}[{\cite[Definition 2.1]{Lin01}}]
    A unital $C^*$-algebra is said to be \emph{tracially approximately finite dimensional} (tracially AF) if for any non-empty finite subset $\mathcal{F} \subset A$ with  $x_1$, any $\varepsilon > 0$ and any $a \in A_+$, there exists a finite dimensional $C^*$-algebra $B \subset A$ with unit $p:=1_B$ such that
    \begin{enumerate}[label=(\roman*)]
        \item $px \approx_\varepsilon xp$ for all $x \in \mathfrak{F}$,
        \item $\mathrm{dist} (pxp, B) < \varepsilon$ for all $x \in \mathfrak F$, and $\| px_1p\| \geq \|x_1\| - \varepsilon$,
        \item $1_A-p \precsim a$ .
    \end{enumerate}
    A non-unital $C^*$-algebra is tracially AF if its minimal unitization is tracially AF.
\end{defn}

\begin{rmk}
    By \cite[Proposition 3.8]{Lin01}, conditions (ii) and (iii) are equivalent to 
    \begin{enumerate}[label=(\roman*'), start=2]
        \item $pxp \in_\varepsilon B$ for all $x \in \mathfrak F$,
        \item $1-p$ is unitarily equivalent to a projection in $\overline{aAa}$.
    \end{enumerate}
    In fact, the most common definition of simple unital tracially AF algebra in the literature uses condition (ii') rather than (ii). 
\end{rmk}

\begin{rmk}
Every simple, separable, tracially AF $C^*$-algebra has real rank zero, stable rank one and it is also quasidiagonal (\cite[Theorem 3.4]{Lin01}, \cite[Corollary A.22]{FG20}). Additionally, they have strict comparison (as in Definition \ref{def:strictcomp}).
By \cite[Proposition 2.4, Theorem 3.6]{Lin01}, a simple, unital $C^*$-algebra $A$ is tracially AF if and only if for any $\varepsilon > 0$, $\delta > 0$, and any finite subset $\mathfrak F \subset A$, there exists a finite dimensional subalgebra $B \subset A$ with unit $p:=1_B$ such that condition (i), (ii') holds and 
\begin{enumerate}[label=(\roman*), start=4]
    \item $\tau(1-p) < \delta$ for all normalised quasitraces $\tau$ on $A$.
\end{enumerate}
We will make free use of these properties throughout this section.
\end{rmk}

We record the following technical proposition that will be a key ingredient in the main proof of this section. This proposition highlights important structural tracial approximations for tracially AF algebras by $C^*$-algebras of the form $\bigoplus_{k=1}^r M_{n_k}(\mathbb{C}) \otimes C_k$. 
Its proof is contained in \cite[Proposition 6.9]{FLL22}, but for completeness we include it here.

\begin{prop}[{\cite[Proposition 6.9]{FLL22}}]\label{prop:esstracZstab}
    Let $A$ be simple, unital, tracially AF $C^*$-algebra. Then for any finite subset $\mathfrak{F}\subset A$, any positive contraction $a \in A$ and any $\varepsilon>0$, there exist $r\in\mathbb N$ and for each $k\in\{1,...,r\}$ a matrix algebra $M_{n_k}(\mathbb{C})$ and a simple, infinite dimensional, unital, AF algebra $C_k$ such that $B := \bigoplus_{k=1}^r M_{n_k}(\mathbb{C}) \otimes C_k \subset A$ and, writing $p:=1_B$,
    \begin{enumerate}
        \item $px \approx_\varepsilon xp$ for all $x \in \mathfrak F$,
        \item $pxp \in_\varepsilon \bigoplus M_{n_k}(\mathbb{C}) \otimes 1_{C_k}$,
        \item $1-p \precsim a$.
    \end{enumerate}
\end{prop}

\begin{proof}
	Let $\mathfrak F \subset A$ be a finite subset, $a\in A$ a positive contraction and  $\varepsilon>0$. By the TAF condition, there is a finite dimensional subalgebra $F:= \bigoplus_{s=1}^r M_{_{n_s}} (\mathbb C)$ with $p=1_F$ such that
	\begin{enumerate}
		\item $xp \approx_\varepsilon px$ for all $x \in \mathfrak{F}$,
		\item $pxp \in_\varepsilon F$, and
		\item $1-p \precsim a$.
	\end{enumerate}
	As in the proof of \cite[Proposition 6.9]{FLL22}, there is a unital, simple, infinite dimensional, AF algebra $B_s$ with a unital embedding $\phi_s : B_s \to \overline{e_{11}^{(s)}Ae_{11}^{(s)}}$ such that the induced map at the level of Murray--von Neumann semigroups $V(\phi_s): V(B_s) \to V(\overline{e_{11}^{(s)}Ae_{11}^{(s)}})$ is surjective.
	Set $C_s := \phi_s(B_s) \subset \overline{e_{11}^{(s)}Ae_{11}^{(s)}}$ and $B := \bigoplus_{s=1}^r M_{n_s}(\mathbb{C}) \otimes C_s$. Observe that $B$ sits as a subalgebra of $A$ via the following map
	\begin{equation}
        \bigoplus_{s=1}^r M_{n_s}(\mathbb{C}) \otimes C_s \to A, \qquad e_{ij}^{(s)}\otimes c \mapsto e_{i1}^{(s)} \cdot c \cdot e_{1j}^{(s)}. 
	\end{equation}
	This map satisfies that $e_{ij}^{(s)} \otimes 1_{C_s} \mapsto e_{ij}^{(s)}$. Then under this identification, it follows that $pxp \in_\varepsilon \bigoplus_{s=1}^r M_{n_s}(\mathbb{C}) \otimes 1_{C_s}$. The last part of the statement is the content of \cite[Proposition 6.9]{FLL22}.
\end{proof}

\subsection{Generator rank of tracially AF algebras}

After introducing all the necessary notions, we focus on the main theorem of this section. We will show that tracially AF algebras have generator rank at most $1$. 

Before proceeding, let us briefly explain the ideas behind the proof. Given any two self-adjoint elements $a$ and $b$, and and arbitrary element $c$, we use Proposition \ref{prop:esstracZstab} to produce a tracial approximation of the form 
$${\bigoplus_{s=1}^r M_{n_s}(\mathbb{C}) \otimes C_s}$$ 
whose unit we denote by $p$. It is a crucial ingredient that compressions $pap, pbp$ and $pcp$ are approximately in $\bigoplus_{s=1}^r M_{n_s}(\mathbb{C}) \otimes 1_{C_s}$. This enables us to produce elements that almost commute with $a,b$ and $c$. 

Standard matricial techniques, similar to those used to prove single generation and generator rank for UHF-stable  $C^*$-algebras, allow us to 
perturb $ap$ and $pb$ to produce elements $x$ and $y$ that generate $\bigoplus_{k=s}^r M_{n_s}(\mathbb{C}) \otimes C_s$ and approximately recover $cp$. 
The main difficulty lies in simultaneously approximating the tracial remainder $c(1-p)$. 

To do this, we will produce a new tracial approximation. In each AF-algebra $C_s$, we choose a matrix subalgebra of dimension $5$ whose unit is tracially close to that of $C_s$. 
Adding these elements produces a subprojection $f$ of $p$ supporting a finite dimensional algebra of the form $\bigoplus_{k=s}^r M_{n_s}(\mathbb{C}) \otimes M_5(\mathbb C)$. 
While constructing such $f$, we can also arrange $1-f$ to be tracially small and for $f$ to approximately commute with $a, b$ and $c$.

We fix some projection $q \in \bigoplus_{k=s}^r M_{n_s}(\mathbb{C}) \otimes M_5(\mathbb C)$ which, by strict comparison, satisfies $1-f \precsim q$. 
This yields a partial isometry $v$ with support projection $1-f$ and range projection in $qAq$. 
We can encode $v$ and $v(1-f)c(1-f)v^*$ into one generator so that $(1-f)c(1-f)$ is arbitrarily close to the $C^*$-algebra generated by these new generators. Since $f$  approximately commutes with $c$, the off-diagonal term $fc(1-f)$ is small in norm, so that the distance from $c(1-f)$ to the generated subalgebra is also small.
One further detail remains to be addressed. Although the original tracial approximation satisfies that $cp$ is close to the larger subalgebra $\bigoplus_{k=s}^r M_{n_s}(\mathbb{C}) \otimes 1_{C_s}$, it does not follow that
$cf$ is actually close to the finite dimensional subalgebra $\bigoplus_{k=s}^r M_{n_s}(\mathbb{C}) \otimes M_5(\mathbb C)$. However, we can enforce this by perturbing our generators one more time.

We choose two orthogonal copies of $M_2(\mathbb C)$ via four pairwise orthogonal and equivalent projections $q_2, q_3, q_4, q_5 \in \bigoplus_{k=s}^r 1_{M_{n_s}} \otimes M_5(\mathbb C)$ together with partial isometries $u$ and $w$ witnesing $q_2 \sim q_3$ and $q_4 \sim q_5$. We then encode $cw$ with one of our generators to approximately recover $cq_2$. By additionally encoding $u$ in the same generator, we can in fact recover $cf$.

Thus both $cf$ and $c(1-f)$ and hence $c = cf + c(1-f)$ can be approximated by elements in the $C^*$-algebra generated by the perturbations of $a$ and $b$. Since the perturbation can be arbitrarily small and the element $c$ is also arbitrary, we then obtain that the generator rank of $A$ is at most $1$.

Our argument requires three mutually orthogonal and complementary corners: one containing a projection that controls $1-f$ while the other two containing the pairs $q_2, q_3$ and $q_4, q_5$, respectively. 
The orthogonality of these corners ensures that mixed products involving the corresponding partial isometries vanish. This explains why we need the matrix size of the second tensor factor to be at least $5$.

We begin with a technical lemma that will produce tracial factorizations of the form $\bigoplus M_{n_s} (\mathbb C) \otimes M_5(\mathbb C)$.

\begin{lemma}\label{lem:M_5.tracial.approx}
    Let $A$ be a simple, separable, tracially AF $C^*$-algebra. Let $\varepsilon, \delta > 0, a,b \in A_{sa}$ and $c \in A$. 
    Then there exists a subalgebra $\bigoplus_{s=1}^r M_{n_s}(\mathbb C) \otimes M_5(\mathbb C) \subset A$ with unit $f$ such that for all $x \in \{ a,b, c \}$, we have
    \begin{enumerate}
        \item $fx \approx_\varepsilon xf$,

        \item $fxf \in_\varepsilon \bigoplus_{s}^r M_{n_s}(\mathbb C) \otimes 1_{M_s}$ 

        \item $\tau(1-f) < \delta$ for all $\tau \in T(A)$.

    \end{enumerate}
\end{lemma}

\begin{proof}
    Set $\mathfrak F = \{a, b, c, c^*\}$. We can assume that $0 < \delta \leq \varepsilon/4$.
    Thus, by Proposition \ref{prop:esstracZstab}, there is a finite dimensional $C^*$-algebra $\bigoplus\limits_{s=1}^r M_{n_s}(\mathbb{C})$ and simple, infinite dimensional, unital, AF algebras $C_s$ such that $\bigoplus\limits_{s=1}^r M_{n_s}(\mathbb{C}) \otimes C_s \subset A$ and, denoting its unit by $p$, we have
    \begin{enumerate}
        \item $p x \approx_{\delta} x p$ for all $x \in \mathfrak F$,
        \item there exist $a_0, b_0, c_0 \in \bigoplus M_{n_s}(\mathbb{C}) \otimes 1_{C_s}$ such that 
        \begin{equation}
        	p a p \approx_\delta a_0, \qquad p b p \approx_\delta b_0, \qquad p c p \approx_\delta c_0, \notag 
        \end{equation}
  
        \item $\tau(1-p) < \delta/2$ for all $\tau \in T(A)$.
    \end{enumerate}
    Since each $C_s$ is an infinite-dimensional, simple, unital, AF algebra, there is a copy of $M_{5}(\mathbb C)$ within $C_s$ in such a way that its unit is uniformly close to $1_{C_s}$ in trace. Precisely, we will denote this copy contained in $C_s$ by $M_5^{(s)}$ and write its unit as $f_s$. Then we can assume that 
    \begin{equation} \label{eq:fs_largetrace}
    \tau(1_{C_s}- f_s) < \delta/2 \qquad \text{for all} \qquad \tau \in T(C_s). 
    \end{equation}
    Observe that for any $\tau \in T(A)$ and $1 \leq s \leq r$, simplicity yields $\tau(1_{M_s} \otimes 1_{c_s}) > 0$. It follows that 
    \begin{equation}
        \tau_s (x) = \tau(1_{M_{n_s}} \otimes 1_{C_s})^{-1} \tau(1_{M_{n_s}} \otimes x), \qquad x\in C_s,
    \end{equation}
    is a tracial state on $C_s$.
    We define
    \begin{equation}
	   f:= \bigoplus_{s=1}^r 1_{M_{n_s}} \otimes f_s. 
    \end{equation}
    Thus, for $\tau \in T(\bigoplus_s M_{n_s} \otimes C_s)$,
    \begin{equation}
        \begin{aligned}
        \tau(p-f) &= \tau\left( \bigoplus_{s=1}^r 1_{M_s} \otimes (1_{C_s} - f_s) \right) \\ 
        & = \sum_{s=1}^r \tau(1_{M_s} \otimes 1_{C_s} ) \tau_s( 1_{C_s} - f_s) \\
        & < \sum_{s=1}^r \tau(1_{M_s} \otimes 1_{C_s} ) \delta/2 = \delta/2.
        \end{aligned}
    \end{equation}
    By a similar argument, we obtain that for any $\tau \in T(A)$
    \begin{equation}
        \tau(p-f) = \tau(p) \left( \tau(p)^{-1} \tau(p-f) \right) < \delta/2. 
    \end{equation}
    This implies
    \begin{equation}
	\tau(1-f) =  \tau(1-p) + \tau(p-f) < \delta.
\end{equation}
        
Since $a_0, b_0, c_0 \in \bigoplus_{s=1}^r M_{n_s}( \mathbb{C}) \otimes 1_{C_s}$, these elements commute with $f$. Then, using that $f \leq p$, we have
\begin{equation} \label{eq:fF_approxcommute}
    fa = fpa \approx_{\delta} fpap \approx_{\delta} fa_0   = a_0 f \approx_{\delta} pap f \approx_{\delta} apf = af
\end{equation}
and
\begin{equation}
    faf = fpapf \approx_{\varepsilon} fa_0 f \in \bigoplus_{s=1}^r M_{n_s}(\mathbb C) \otimes 1_{M_s}.
\end{equation}
By the same argument, $f x \approx_{\varepsilon} x f$ and $fxf \in_\varepsilon \bigoplus_{s=1}^r M_{n_s}(\mathbb C) \otimes 1_{M_s}$ for any $x \in \mathfrak F$.
\end{proof}

\begin{thm}\label{main.thm}
    Let $A$ be a simple, separable, tracially AF $C^*$-algebra. Then the generator rank of $A$ is at most $1$. In particular, $A$ is singly generated.
\end{thm}

\begin{proof}
Let us begin by assuming that $A$ is unital. 
Let $\varepsilon > 0$, $a,b \in A_{sa}$ and $c \in A$. 
According to the definition of generator rank, we will construct $x,y \in A_{sa}$ such that 
\begin{equation}\label{mainthm:eq.mu}
	a\approx_\varepsilon x, \quad b \approx_\varepsilon y, \quad \mathrm{dist}(c,C^*(x,y))< \varepsilon.
\end{equation} 

Set $\mathfrak F = \{a, b, c, c^*\}$ and $0 < \delta < \min \{\frac{\varepsilon}{15}, \frac{1}{12}\}$.
Thus, by Lemma~\ref{lem:M_5.tracial.approx}, there exists a subalgebra $\bigoplus_{s=1}^r M_{n_s}(\mathbb C) \otimes M_5^{(s)} \subset A$ with unit $f$ such that $M_5^{(s)} \cong M_5(\mathbb C)$ and
    \begin{enumerate}
        \item $fh \approx_\delta hf$ for all $h \in \mathfrak F$,

        \item there exist $a_0, b_0, c_0  \in \bigoplus_{s}^r M_{n_s}(\mathbb C) \otimes 1_{M_s}$ such that 
        $$ 
        faf \approx_\delta a_0, \qquad fbf \approx_\delta b_0, \qquad fcf \approx_\delta c_0,
        $$

        \item $\tau(1-f) < \delta$ for all $\tau \in T(A)$.

    \end{enumerate}
Let $(e_{ij}^{(s)})$ be the matrix units of $M_{5}^{(s)} $ and for $1 \leq i \leq 5$, set 
\begin{equation}
	\tilde{e}^{(s)}_{ij}:= 1_{M_{n_s}} \otimes e^{(s)}_{ij},   \qquad q_i := \bigoplus_{s=1}^r \tilde{e}^{(s)}_{ii}.
\end{equation}
It follows immediately that $(\tilde{e}_{ij}^{(s)})$ also satisfies the matrix unit relations.
We also define
\begin{equation}
    u := \bigoplus_{s=1}^r \tilde{e}_{12}^{(s)} \qquad \text{and} \qquad w:= \bigoplus_{s=1}^r \tilde{e}_{34}^{(s)}. 
\end{equation}
These are partial isometries satisfying 
\begin{equation}
	uu^* = q_1, \qquad u^*u = q_2, \qquad ww^* = q_3, \qquad w^* w = q_4.
\end{equation}

On the other hand, by setting $e:=1-f$, we have
\begin{equation}
	\tau(e) < \delta  \leq \frac{1}{6}.
\end{equation}
Then $5 \tau(e) < 1 - \tau(e)$ which yields
\begin{equation}
      \tau(e) < \frac{1-\tau(e)}{5} = \frac{\tau(f)}{5} = \tau(q_1).
\end{equation}
Then by strict comparison, there is $v \in q_1 Ae$ such that
\begin{equation}
	v^*v = e \qquad \text{and} \qquad q := vv^* \leq q_1.
\end{equation}
We define
\begin{equation}
	d := vcv^* = qv c v^* q \in qAq.
\end{equation}
Thus 
\begin{equation}
    v^* d v = v^* vcv^*v = ece.
\end{equation}

Since $eAe$ has real rank zero, we can find a self-adjoint element $a_1 \in eAe$ with finite spectrum such that $a_1 \approx_{\delta} eae$. 
After an arbitrarily small perturbation, we can assume $a_0$ and $a_1$ have disjoint spectra and are invertible.
We can write $a_0 = (a_0^{(s)}) \in \bigoplus_{s=1}^r M_{n_s}(\mathbb{C}) \otimes 1_{C_s}$. As before, after a further arbitrarily small perturbation, we can assume that the elements $a_0^{(s)}$ have pairwise disjoint spectra.
Let $0 <\mu < \delta$ be smaller than the distance between any two distinct elements in $\sigma(a_0)\cup\sigma(a_1)\cup\{0\}$.
With all this setup, we can now define our first generator:  
\begin{equation}\label{mainthm:eq.xn}
		x := a_1 + \sum_{s=1}^r \sum_{k=1}^5 \left( a_0^{(s)} + \frac{k \mu}{5} \right) \otimes e_{kk}^{(s)}.
\end{equation}
In order to somewhat simplify our notation, let us define
\begin{equation}
	z: = \bigoplus_{s=1}^r \sum_{j=1}^4 (\tilde{e}_{5j}^{(s)} + \tilde{e}_{j5}^{(s)}).
\end{equation}
We choose 
\begin{equation}
	0 < \lambda_1 , \ldots , \lambda_4 < \min\left\{\frac{\delta}{2}, \frac{\delta}{2(\|c\|+\delta)} \right\}    
\end{equation}
and define
\begin{equation}
	y:= ebe + b_0 f + \lambda_1 z +\lambda_2 \left( v + v^* \right) + \lambda_3 \left(du + u^* d^* \right) + \lambda_4 (c_0 w + w^*c_0^* ).
\end{equation}
Our choice of $\lambda_i$ guarantees that the terms in $y$ with the corresponding coefficient $\lambda_i$ have norm smaller than $\delta$.
By construction, $x$ and $y$ are self-adjoint elements. We define $B:= C^*(x,y)$.

\begin{claim}
	The generators of $B$ satisfy $x \approx_\varepsilon a$ and $y \approx_\varepsilon b$.
\end{claim}

\begin{proof}[Proof of claim]
Let us verify first that $x$ is close to $a$. Indeed,
\begin{equation} 
\begin{aligned}
	x-a & = x - eae - fae - eaf - faf \\
	& \approx_{8\delta} x-eae-faf \\
	& = x - eae - fpapf \\
	& \approx_{\delta} a_1 + \sum_{s=1}^r \sum_{k=1}^5 \left( a_0^{(s)} + \frac{k \mu}{5} \right) \otimes e_{kk}^{(s)} -eae - fa_0f \\
	& \approx_{\delta} \sum_{s=1}^r \sum_{k=1}^5 \left( a_0^{(s)} + \frac{k \mu}{5} \right) \otimes e_{kk}^{(s)} - \sum_{s=1}^r \sum_{k=1}^5 a_0^{(s)} \otimes e_{kk}^{(s)} \\
	& = \sum_{s=1}^r \sum_{k=1}^5 \frac{k \mu}{5} \otimes e_{kk}^{(s)} \\
	& \approx_{5 \delta} 0,
\end{aligned}
\end{equation}
where we used \eqref{eq:fF_approxcommute} on the second line and the definitiond of $a_0$ and $a_1$ on the fourth and  fifth lines, respectively.

On the other hand, the second generator satisfies
\begin{equation} 
\begin{aligned}
	y-b & \approx_{4 \delta} ebe + b_0 f - (ebe+fbe+ebf+fbf)  \\
	& \approx_{8\delta} ebe + b_0 f - ebe - fbf \\
	& \approx_{\delta} 0,
\end{aligned}
\end{equation} 
where we used \eqref{eq:fF_approxcommute} on the second line and the definition of $b_0$ on the third.
Our proposed generators then satisfy $a \approx_\varepsilon x$ and $b \approx_\varepsilon y$.
\end{proof}

\begin{claim}
	The elements $e, \tilde{e}_{ii}^{(s)}$ and $q_i$ belong to $B$ for $1 \leq i \leq 5$.
\end{claim}

\begin{proof}[Proof of claim]
Our choice of $\mu$ and the fact that $a_1$ and all $a_0^{(s)}$ are invertible guarantees that using a suitable continuous function we have $e$ and $\tilde{e}_{ii}^{(s)}$ are in $C^*(x) \subset B$ for $1 \leq s \leq r$ and $1 \leq i \leq 5$.
It then follows that 
\begin{equation}
    q_i = \bigoplus_{s=1}^r \tilde{e}_{ii}^{(s)}  \in C^*(x) \subset B \qquad  \text{for} \qquad 1 \leq i \leq 5. \qedhere
\end{equation}
\end{proof}

\begin{claim}\label{claim:e5i}
	The elements $\tilde{e}_{5i}^{(s)}$ and $\tilde{e}_{i5}^{(s)}$ belong to $B$ for $1 \leq i \leq 4$ and $1 \leq s \leq r$. 
\end{claim}

\begin{proof}[Proof of claim]
This claim follows from computing $\tilde{e}_{55}^{(s)} y \tilde{e}_{jj}^{(s)}$ for $1 \leq j \leq 4$. Since our expression for $y$ is rather long, we break this down into several smaller calculations.

First, since $b_0 \in \bigoplus_{s=1}^r M_{n_s}(\mathbb{C}) \otimes 1_{C_s}$, it commutes with all $\tilde{e}_{ij}^{(s)}$ and we have the following,
\begin{equation}
    \tilde{e}_{ii}^{(s)} \cdot b_0 f \cdot \tilde{e}_{jj}^{(k)}  = \delta_{ks} b_0 \cdot \tilde{e}_{ii}^{(s)} \tilde{e}_{jj}^{(s)} = \delta_{ks} \delta_{ij} b_0 \cdot \tilde{e}_{ii}^{(s)}. 
\end{equation}
Likewise, 
\begin{equation}
	\begin{aligned}
   \tilde{e}_{ii}^{(s)} \cdot z \cdot \tilde{e}_{jj}^{(k)}& = \tilde{e}_{ii}^{(s)} \left( \bigoplus_{l=1}^r \sum_{m=1}^4 (\tilde{e}_{5m}^{(l)} + \tilde{e}_{m5}^{(l)}) \right)  \tilde{e}_{jj}^{(k)} \\
    & = \sum_{m=1}^4 \left( \delta_{i5} \tilde{e}_{5m}^{(s)} \tilde{e}_{jj}^{(k)} + e_{i5}^{(s)}e_{jj}^{(k)}\right) \\
    &= \delta_{ks} \left(\delta_{i5} \tilde{e}_{5j}^{(s)} + \delta_{j5} \tilde{e}_{i5}^{(s)} \right). 
\end{aligned}
\end{equation}
Also, using that $v \in qAe$, that $\tilde{e}_{ii}^{(s)} \leq f$, and that $e$ and $\tilde{e}_{ii}^{(s)}$ are orthogonal, we get
\begin{equation}
    \tilde{e}_{ii}^{(s)} \left( v + v^* \right) \tilde{e}_{jj}^{(k)} = 0.
\end{equation}
In the same manner,  since $q \leq q_1$ and $q_1$ is orthogonal to $q_j$ for all $j \neq 1$, it follows that $q$ is orthogonal to $q_j$ for $j \neq 1$. Using that $d \in qAq$, we obtain
\begin{equation}
	\begin{aligned} 
    \tilde{e}_{ii}^{(s)} \left( du \right) \tilde{e}_{jj}^{(k)} & = \tilde{e}_{ii}^{(s)} \cdot d \left( \bigoplus_{l=1}^r  \tilde{e}_{12}^{(l)} \right) \tilde{e}_{jj}^{(k)}  \\
    & = \delta_{2j} \tilde{e}_{ii}^{(s)} \cdot d \cdot \tilde{e}_{12}^{(k)}  \\
    & = \delta_{2j} \tilde{e}_{ii}^{(s)} \cdot q_1 d \cdot \tilde{e}_{12}^{(k)}  \\
    & =  \delta_{2j} \delta_{1i} \tilde{e}_{11}^{(s)} \cdot d \cdot \tilde{e}_{12}^{(k)}. 
\end{aligned}
\end{equation}
Similarly,
\begin{equation}
    \tilde{e}_{ii}^{(s)} \left( u^*d^* \right) \tilde{e}_{jj}^{(k)} 
    = \delta_{2i} \delta_{1j} \tilde{e}_{21}^{(k)} \cdot d^* \cdot \tilde{e}_{11}^{(s)}.
\end{equation}
Proceeding as above, using that $c_0 \in \bigoplus_{s=1}^r M_{n_s}(\mathbb C) \otimes 1_{C_s}$, which therefore commutes with all $\tilde{e}_{ij}^{(s)}$, we obtain
\begin{equation} 
\begin{aligned}
    \tilde{e}_{ii}^{(s)} \left( c_0 w \right) \tilde{e}_{jj}^{(k)} & = c_0 \cdot \tilde{e}_{ii}^{(s)} \cdot \left( \bigoplus_{l=1}^r \tilde{e}_{34}^{(l)} \right) \cdot \tilde{e}_{jj}^{(k)} \notag \\
    & = \delta_{3i} c_0 \cdot \tilde{e}_{34}^{(s)} \cdot \tilde{e}_{jj}^{(k)} \notag \\
    & = \delta_{3i} \delta_{4j} \delta_{ks} \cdot c_0 \cdot \tilde{e}_{34}^{(s)} 
\end{aligned}
\end{equation}
and
\begin{equation}
    \tilde{e}_{ii}^{(s)} \left( w^*c_0^* \right) \tilde{e}_{jj}^{(k)} = \delta_{4i} \delta_{3j} \delta_{ks} \cdot c_0^* \cdot \tilde{e}_{43}^{(s)}.
\end{equation}
Combining these computations and using that $e$ is orthogonal to $\tilde{e}_{ii}^{(s)}$, we find for all $1 \leq j \leq 4$ that
\begin{equation} 
\begin{aligned}
    0 & = \tilde{e}_{55}^{(s)} \left( ebe \right) \tilde{e}_{jj}^{(s)} \\ 
    & = \tilde{e}_{55}^{(s)} \left( b_0 f \right) \tilde{e}_{jj}^{(s)}  \\
    & = \tilde{e}_{55}^{(s)} \left( v + v^*\right) \tilde{e}_{jj}^{(s)}   \\
    & = \tilde{e}_{55}^{(s)} \left( du + u^* d^*\right) \tilde{e}_{jj}^{(s)}  \\
    & = \tilde{e}_{55}^{(s)} \left( c_0 w + w^* c_0^* \right) \tilde{e}_{jj}^{(s)}.
\end{aligned}
\end{equation}
Hence, for all $1 \leq j \leq 4$ and $1 \leq s \leq r$,
\begin{equation}
    \lambda_1 \tilde{e}_{5j}^{(s)} = \tilde{e}_{55}^{(s)} \cdot y \cdot \tilde{e}_{jj}^{(s)} \in B.
\end{equation}
This proves the claim. 
\end{proof}

\begin{claim}
	The elements $u, v$ and all the matrix units $\tilde{e}_{ij}^{(s)}$ belong to $B$.
\end{claim}
\begin{proof}[Proof of the claim]
The matrix units belong to $B$ by the following calculation:
\begin{equation}
    \tilde{e}_{ij}^{(s)} = \tilde{e}_{i5}^{(s)} \tilde{e}_{5j}^{(s)} \in B.
\end{equation}
Since $u$ and $w$ are spanned by elements of this form, we obtain that $u$ and $w$ are also elements in $B$.

We proceed to show that $v$ also belongs to $B$. To this end, we calculate $q_1 y e$.
By construction, $d, q, q_1, u, w, z \in fAf$. Since $e$ and $f$ are orthogonal we obtain
\begin{equation}
    q_1 ebe = q_1 b_0 f e = q_1 z e = q_1 (du + u^* d^*) e = 0.
\end{equation}
On the other hand, since $c_0$ commutes with $q_1$ and $w$, the previous argument yields
\begin{equation}
	q_1 (c_0 w + w^* c_0^*) e = c_0^* q_1 w^* e = 0.
\end{equation}
Combining these calculations with $v \in q_1Ae$ yields
\begin{equation}
    q_1 y e = \lambda_2 q_1 \left( v + v^* \right) e = \lambda_2 q_1  v e = \lambda_2 v.
\end{equation}
This shows that
\begin{equation}
    v = \lambda_2^{-1} q_1 y e \in B. 
\end{equation}
This completes the proof of the claim.
\end{proof}

\begin{claim}
	The elements $d, ece$ and $c_0f$ belong to $B$.
\end{claim}
\begin{proof}[Proof of claim]
Proceeding as before, we compute $q_1 y q_2$ to deduce that $d$ lies in $B$.
Applying the calculations from the proof of Claim \ref{claim:e5i}, we obtain
\begin{equation} 
	0 = q_1 ebe q_2 = q_1 b_0 f q_2 = q_1 z q_2 = q_1 (v + v^*) q_2 = q_1 \left( c_0 w + w^* c_0^* \right) q_2.
\end{equation}
As $d\in q_1Aq_1$ and $u\in q_1Aq_2$, we have
\begin{equation}
    q_1 y q_2
    = \lambda_3 q_1 \left( du + u^* d^* \right) q_2 = \lambda_3 q_1 d u q_2 = \lambda_3 d u. 
\end{equation}
Using that $uu^* = q_1$, $q \leq q_1$ and $d \in qAq$, we obtain
\begin{equation}
    d = \lambda_3^{-1} q_1 y q_2 u^* \in B. 
\end{equation}
In particular, this implies 
\begin{equation}
    ece = v^* d v \in B.
\end{equation}
Using similarly that
\begin{equation}
    q_3 e = q_3 b_0 f q_4 = q_3 z q_4 = q_3  \left( du + u^* d^* \right) q_4 = 0,
\end{equation}
we obtain
\begin{equation}
     q_3 y q_4 = \lambda_4 q_3 \left( c_0 w + w^* c_0^* \right) q_4 = \lambda_4 c_0 w.
\end{equation}
This yields
\begin{equation}
    c_0 q_3 = c_0 ww^* = \lambda_4^{-1} q_3 y q_4 w^* \in B.
\end{equation}
By conjugating $c_0q_3$ by $\bigoplus_{s=1}^r 1_{M_s}\otimes e^{(s)}_{3i} \in B$ for all $1\leq i \leq 5$ and noting that these elements commute with $c_0 \in \bigoplus_{s=1}^r M_{n_s} \otimes 1_{C_s}$, it follows that
\begin{equation}
    c_0q_i = \left(\bigoplus_{s=1}^r 1_{M_s}\otimes e^{(s)}_{i3}\right) c_0 q_3 \left(\bigoplus_{s=1}^r 1_{M_s}\otimes e^{(s)}_{3i}\right) \in B.
\end{equation}
This in turn implies
\begin{equation}
    c_0 f = \sum_{i=1}^5 c_0 q_i \in B.
\end{equation}
This completes the proof of the claim.
\end{proof}

We will finish the proof of the theorem in the unital case by showing that $c$ is close to $B$. Indeed, by \eqref{eq:fF_approxcommute}, we have
\begin{align}
    c = fcf + ecf + fce + ece \approx_{8\delta} fcf + ece \approx_{\delta} c_0 f + ece \in B.
\end{align}
Therefore, we have shown that $x \approx_\varepsilon a$ and $y \approx_\varepsilon b$ and $c \in_\varepsilon B$. This proves that the generator rank of $A$ is at most 1.

If $A$ is not unital, we use that $A$ has an approximate unit $(e_n)$ consisting of projections. Then, we set $A_n:= e_n A e_n$, which is a unital, simple $C^*$-algebra. Since $A_n$ is a hereditary subalgebra of $A$, it follows that $A_n$ is also tracially AF with real rank zero and strict comparison (\cite{Lin01}). 
Then, by the above argument, we obtain that $\mathrm{gr}(A_n) \leq 1$ for all $n \in \mathbb{N}$. By Proposition \ref{thm:gr.ind.limits}, we finally obtain $\mathrm{gr}(A) \leq 1$.
\end{proof}

We point out that this theorem actually produces new examples of singly generated $C^*$-algebras. For instance, $C^*$-algebras in the class of non-$\mathcal {Z}$-stable and non-nuclear tracially AF algebra constructed in \cite{NW21} are singly generated.


\end{document}